\documentclass[12pt]{amsart}
\usepackage{setspace}
\usepackage{amsmath, amssymb,amsthm, amsfonts, latexsym}
\newtheorem{thm}{Theorem}[section]
\newtheorem{lemma}{Lemma}[section]
\newtheorem{prop}{Proposition}[section]

\newtheorem{df}{Definition}[section]

\theoremstyle{remark}
\newtheorem{remark}{Remark}[section]

\newcommand{\Ric}{\mbox{Ric}}
\newcommand{\R}{\mathbb R}

\numberwithin{equation}{section}
\newcommand{\be}{\begin{equation}}
\newcommand{\ee}{\end{equation}}
\def\p{\partial}

\def\la{\langle}
\def\ra{\rangle}
\def\lf{\left}
\def\ri{\right}
\def\Pi{\displaystyle{\mathbb{II}}}
\def\Ric{\text{\rm Ric}}
\def\e{\epsilon}

\def\L{\mathcal{L}}

\def\vh{\vspace{.2cm}}

\def\m{\mathfrak{m}}
\def\mS{\mathbb{S}}

\def\dvo{dv_{\sigma_0}}
\def\dvr{dv_{\sigma_r}}

\def\dvk{dv_{\sigma_k}}

\def\a{\alpha}
\def\b{\beta}
\def\bb{\bar{b}}

\def\ol{\overline}

\def\sumi{\sum_{i=1}^3}

\def\bee{\begin{equation*}}
\def\eee{\end{equation*}}

\def\Hbbr{H_{\bar{b}} (r)}

\def\mby{\mathfrak{m}_{_{BY}}}
\def\mwy{\mathfrak{m}_{_{WY}}}
\def\Ewy{E_{_{WY}}}
\def\MK{\mathbb{R}^{3,1}}
\def\S{\Sigma}

\begin{document}

\title[]{On Second variation of Wang-Yau quasi-local energy }

\author{Pengzi Miao}
\address[Pengzi Miao]{Department of Mathematics, University of Miami, Coral Gables, FL 33146, USA.}
\email{pengzim@math.miami.edu}

\author{Luen-Fai Tam$^1$}
\address[Luen-Fai Tam]{The Institute of Mathematical Sciences and Department of
 Mathematics, The Chinese University of Hong Kong, Shatin, Hong Kong, China.}
 \email{lftam@math.cuhk.edu.hk}

\thanks{$^1$Research partially supported by Hong Kong RGC General Research Fund  \#CUHK 403011}

\renewcommand{\subjclassname}{
  \textup{2010} Mathematics Subject Classification}

\date{}

\begin{abstract}
We study a functional on the boundary of a compact Riemannian $3$-manifold of nonnegative scalar curvature.
The functional arises as the second variation of the Wang-Yau quasi-local energy in general relativity.
We prove that the functional is positive definite on large coordinate spheres, and more general on nearly round surfaces
including large constant mean curvature spheres  in asymptotically flat $3$-manifolds with positive mass;  it is also positive definite
on small geodesics spheres, whose centers do not have vanishing curvature, in Riemannian $3$-manifolds of nonnegative
scalar curvature. We also  give examples of functions $H$,
which can be made  arbitrarily close to  $ 2$,
on the standard $2$-sphere $ (\mS^2, \sigma_0)$    such that
 the triple $(\mS^2, \sigma_0, H)$ has positive Brown-York mass while the associated  functional
 is negative somewhere.
\end{abstract}

\maketitle

\markboth{Pengzi Miao and Luen-Fai Tam}
{On second variation of Wang-Yau quasi-local energy}

%\tableofcontents

\section{Introduction}

In \cite{WangYau-PRL,WangYau08}, Wang and Yau introduced a new quasilocal mass.
Briefly speaking, its definition is as follows.
Let  $ \Sigma$ be a closed $2$-surface, in a spacetime $N$ satisfying the dominant energy condition,
such that $\Sigma$ bounds a compact, spacelike hypersurface $ \Omega$.
Denote the induced Riemannian metric on $ \Sigma $ by $ \gamma$.
Given a function $ \tau $ on $ \S$ such  that
 $ \hat{\gamma} = \gamma + d \tau \otimes d \tau $ is a metric  of positive Gaussian curvature,
one considers the isometric embedding
$$X: (\Sigma,\gamma) \hookrightarrow \MK$$
where  $ X = (\hat{X}, \tau)  $  and $ \hat{X} = (\hat{X}_1, \hat{X}_2, \hat{X}_3 )$ is an isometric
 embedding of $ (\Sigma, \hat{\gamma} )$ in $ \R^3 = \{(x,0)\in \R^{3,1}\}$.
 Associated with  each such a function $\tau$ or equivalently each such an isometric
embedding $X$, Wang and Yau  introduced a quantity, which we denote by $\Ewy(\Sigma,\tau)$,
 called the  quasi-local energy of $ \Sigma$ in $ N$ with respect to $\tau$.
The Wang-Yau quasi-local mass of $ \Sigma$ in $ N $ is then defined by
\be
 \mwy(\Sigma) = \inf_{ \tau} \Ewy (\Sigma, \tau)
\ee
where the infimum is taken over all admissible  functions $\tau$ (see   \cite{WangYau08} for
an  exact formula of $ \Ewy(\Sigma, \tau)$  and the definition of admissibility).
It was proved in \cite{WangYau08} that
$\mwy(\Sigma) \ge 0 $  and $ \mwy(\Sigma) = 0$
if  the embedding $\Sigma\hookrightarrow N$ is isometric to  $ \R^{3,1}$  along $\Sigma$.

When $\Sigma $ bounds a time-symmetric $ \Omega$ and $\gamma$ has positive Gaussian curvature,
there is a well-known Brown-York quasi-local mass of $\Sigma$ (\cite{BY1, BY2}) given by
\be
\mby (\Sigma, \Omega) = \frac{1}{8 \pi} \int_\Sigma ( H_0 - H ) \ d v_\gamma
\ee
where $ H_0 $ is the mean curvature of the isometric embedding of $ (\Sigma, \gamma) $
in  $ \R^3$ and $ H $ is  the mean curvature of $ \Sigma$ in $ \Omega$.
In this situation, one has   $ \mby (\Sigma, \Omega) = \Ewy(\Sigma,\tau_0) $, where $\tau_0=0$
is an admissible function and  is also a critical point of $\Ewy(\Sigma,\cdot)$ (\cite{WangYau08}).
The variational definition of $ \mwy (\Sigma) $  suggests   $ \mwy(\Sigma) \le \mby (\Sigma, \Omega)$.
A natural question is whether $\mwy(\Sigma)=\mby (\Sigma, \Omega)$.
(Results regarding
%the global minimizing property of $ \mby(\Sigma, \Omega)$
 the global minimization of $ \Ewy(\Sigma, \cdot) $
   recently have been  announced in \cite{SanYa-Chen}.)

In this paper, we consider the local minimality of $ \mby (\Sigma, \Omega)$.
A main corollary of our result for surfaces in  an asymptotically flat manifold is:

\begin{thm} \label{theorem:  Intro-Large-sphere}
Let $(M, g)$ be an asymptotically flat $3$-manifold.
Let $S_r = \{ x \in M \ | \ | z | =r \}$ be a coordinate sphere in an admissible
 coordinate chart $\{ z_i \}$ on a given end.
 Suppose the ADM mass  of the end is positive. For sufficiently large $r$,
 the Brown-York mass of $S_r$ is a strict local minimum of $\Ewy(S_r,\cdot )$.
\end{thm}

We will prove Theorem  \ref{theorem: Intro-Large-sphere}  by proving Theorem \ref{thm-nearly-round-surface}
in Section \ref{sect: application}  for a larger class of  ``large surfaces", namely
nearly round surfaces near the infinity  which were introduced in \cite{ShiWangWu09}.
As mentioned in \cite{ShiWangWu09},   besides large coordinate spheres,  notable examples of nearly round surfaces in an asymptotically flat $3$-manifold
 include  the constant mean curvature surfaces  constructed in \cite{HuiskenYau96} and \cite{Ye96}.

For small geodesic spheres in a  manifold of nonnegative scalar curvature, we have:

\begin{thm}  \label{thm:intro-small-spheres}
Let $(M, g)$ be a  Riemannian $3$-manifold of  nonnegative scalar curvature. Let  $ p \in M $ be  a point.
For   $r>0$, let $S_r$ be the geodesic sphere of radius $r$ centered at $p$ and  $B_r$ be the corresponding geodesic ball.
If
\be \label{eq-small-BYmass}
\lim_{r\to0}r^{-5}\mby(S_r,B_r)>0,
\ee
then $\mby(S_r,B_r)$ is a strict local minimum of $\Ewy(S_r,\tau)$ for $r>0$ sufficiently small.
\end{thm}

We note that  condition  \eqref{eq-small-BYmass} in Theorem \ref{thm:intro-small-spheres}  is equivalent to
\begin{enumerate}

\item[(i)]  $R(p)>0$,  or

\vh

\item[(ii)] $R(p)=0$ and $|\Ric (p)|^2 >0$,  or

\vh

\item[(iii)]  $R(p)= 0 $, $|\Ric(p)| =0,$ and $ (\Delta R) (p)>0$

\end{enumerate}
which   follows from  the asymptotic expansion
 of $\mby (S_r, B_r) $  in \cite{FanShiTam07} and the assumption $ R \ge 0$.
 Here $ R $, $ \Ric $ denote the scalar curvature, the Ricci curvature of $g$.

For $ \mby(\Sigma, \Omega) = \Ewy (\Sigma, \tau_0 )$ to  locally minimizes $ \Ewy(\Sigma, \cdot)$,
the second variation of $ \Ewy(\Sigma, \cdot)$ at $ \tau_0$ is necessarily nonnegative.
We recall the following result from \cite{MiaoTamXie11}.

\begin{thm}[\cite{MiaoTamXie11}] \label{thm-local-minimum}
Suppose $\mby(\Sigma, \Omega)$ is defined for a $2$-surface $ \Sigma$ bounding a time-symmetric hypersurface $ \Omega $ in a spacetime $N$.
The second variation of $ \Ewy (\Sigma, \cdot) $ at $ \tau_0 = 0 $ (up to multiplication by $\frac{1}{8\pi} $) is
\be
F_{\gamma,H}(\eta) =  :\int_\Sigma \lf[ \frac{ ( \Delta \eta )^2 }{ H  } + ( H_0 - H ) | \nabla \eta  |^2 - \Pi_0 ( \nabla \eta ,\nabla \eta)\ri]
dv_{\gamma}
\ee
where  $\Pi_0$ is the second fundamental form of $(\Sigma, \gamma)$ when it is isometrically embedded in $ \R^3$.
If there exists  a constant $\beta>0$ such that
\be \label{eq-intro-F}
F_{\gamma, H} (\eta) \ge  \beta\int_\Sigma (\Delta\eta)^2dv_{\gamma} ,  \ \forall \  \eta \in W^{2,2}(\Sigma),
\ee
then $ \mby(\Sigma, \Omega)$  is a strict  local minimum of $ \Ewy(\Sigma, \cdot)$.
 \end{thm}

Therefore, to obtain  the local minimality of $\mby(\Sigma, \Omega)$, it suffices to study the functional
 $F_{\gamma,H}(\eta)$.  The induced metric and the mean curvature function on  the surfaces $\{ S_r \}$
 in Theorems \ref{theorem: Intro-Large-sphere}  and \ref{thm:intro-small-spheres}  (after rescaling)
 are close to the standard metric $\sigma_0$ on the unit sphere $ \mS^2$ and the constant $ 2$ respectively.   Thus,
 one may ask whether  $F_{\gamma,H}(\eta)$ satisfies \eqref{eq-intro-F} if  the pair $(\gamma, H)$
is sufficiently close to $(\sigma_0, 2)$.

Our first task  in this paper is to derive some sufficient conditions on such a  pair $(\gamma, H)$  so that \eqref{eq-intro-F} is true.
Applying these  sufficient conditions, we can prove  Theorem \ref{thm-nearly-round-surface},
and part of Theorem \ref{thm:intro-small-spheres} which corresponds to cases (i) and (iii) above.

 The other part of Theorem \ref{thm:intro-small-spheres}, which corresponds to case (ii),
 turns out to be  more subtle. We will prove it  using more refined estimation on $(\mS^2, \sigma_0)$
 (see Theorem \ref{thm-small-spheres-bz} and Proposition \ref{prop-FQ-2}).
 Motivated by our proof of this part of Theorem \ref{thm:intro-small-spheres}, we
  also construct examples  to show that on $(\mS^2, \sigma_0)$, there are functions  $H$ which can be arbitrarily close to 2,
  but $F_{\sigma_0,H}(\eta)<0$ for some $\eta$.

We remark that the general validity of  \eqref{eq-intro-F} is of significance in the study of boundary behaviors of compact manifolds
of nonnegative scalar curvature.  If \eqref{eq-intro-F} is always true, it will impose
 a {\em necessary} condition for a positive function $ H$ on $\Sigma$  to arise as the mean curvature of $ \Sigma $ in some
compact Riemannian $3$-manifold  of nonnegative scalar curvature, bounded by $ (\Sigma, \gamma)$.
So far, a major known necessary condition is
 $
 \int_\Sigma (H_0-H) dv_\gamma \ge0
 $
 by the result of \cite{ShiTam02}.  It is worth to note that
our   examples of $ H$  above,  with $ F_{\sigma_0, H} (\eta) < 0$ for some $ \eta$, also satisfies
 $ \int_{\mS^2 } (2 - H) \dvo > 0 $. Thus,
 if the Brown-York mass always locally minimizes the Wang-Yau quasi-local energy in the time-symmetric situation,
 then  \eqref{eq-intro-F} will constitute a new necessary condition.

This  paper is organized as follows. In Section \ref{sect: preliminary},
we collect some lemmas which are to be used frequently in later sections.
In Section \ref{sect: sufficient}, we obtain  sufficient conditions for   \eqref{eq-intro-F} to hold.
In Section \ref{sect: application},
we apply the derived sufficient conditions to prove Theorem \ref{thm-nearly-round-surface}
which implies Theorem \ref{theorem: Intro-Large-sphere}. In Section \ref{sect: small-sphere-2},
we establish the positivity of $ F_{\gamma, H} $ on small geodesic spheres in Theorem \ref{thm-small-sphere-main}
which implies Theorem  \ref{thm:intro-small-spheres}.
There whether the scalar curvature vanishes at the center of a geodesic sphere makes an important
difference in the proof.  A main result related to the case of vanishing scalar curvature at the sphere  center is
 Theorem \ref{thm-small-spheres-bz}, which we prove using a  functional  inequality
on  the standard sphere $ (\mathbb{S}^2, \sigma_0)$ (Proposition \ref{prop-FQ-2})  which may have
independent interest.
In Section \ref{sect: no-fill-in}, we give examples of $H$ on the
standard unit sphere so that $F_{\sigma_0, H}(\eta)<0$  for some $\eta$
while $ \int_{\mS^2} ( 2 - H) \dvo > 0 $.  In the Appendix,
we list some elementary computational results, which are needed in Section \ref{sect: small-sphere-2}.

\section{Preliminaries} \label{sect: preliminary}
Throughout this paper, $\Sigma$ always denotes a closed $2$-surface that is diffeomorphic to  a  $2$-sphere.
Given a metric $\gamma$ of positive Gaussian curvature and a positive function $H$ on $\Sigma$, we let
\be
F_{\gamma, H} (\eta) =  \int_\Sigma \lf[ \frac{ ( \Delta \eta )^2 }{ H  } + ( H_0 - H ) | \nabla \eta  |^2 - \Pi_0 ( \nabla \eta ,\nabla \eta)\ri]dv_{\gamma}
\ee
for any $ \eta \in W^{2,2} (\Sigma)$. Here  $ \Delta$ and $ \nabla $ denote the Laplacian and the gradient on $ (\Sigma, \gamma)$,
$H_0$ and $\Pi_0$  are the mean curvature and
 the second fundamental form of $(\Sigma, \gamma)$ when it is isometrically
 embedded in $ \R^3$, and $d v_\gamma$ is the volume form on $(\Sigma, \gamma)$.
 We also denote the symmetric bilinear form associated  to $F_{\gamma, H}$ by $Q_{\gamma, H}$. Namely
 \be
Q_{\gamma, H} (\eta_1,\eta_2) =  \int_\Sigma \lf[ \frac{   \Delta \eta_1\cdot\Delta\eta_2  }{ H  } + ( H_0 - H ) \la  \nabla \eta_1, \nabla \eta_1\ra - \Pi_0 ( \nabla \eta_1 ,\nabla \eta_2)\ri]dv_{\gamma}.
\ee
All metrics on $ \Sigma$  below will be assumed to be smooth for simplicity.

We  recall  some  basic results  from \cite{MiaoTamXie11}.

\begin{lemma} \label{lem-earlier}
Let  $\gamma $ be a metric  of positive Gaussian curvature on $ \Sigma$.
Let
$ X = (X_1, X_2, X_3) : (\Sigma, \gamma) \rightarrow \R^3 $
be an isometric embedding of $(\Sigma, \gamma) $ in $ \R^3$.
The functional
$$
F_{\gamma, H_0} (\eta) =  \int_\Sigma \lf[ \frac{ ( \Delta \eta )^2 }{ H_0  }  - \Pi_0 ( \nabla \eta ,\nabla \eta)\ri]dv_{\gamma}
$$
satisfies:
\begin{enumerate}
\item[(i)] $ F_{\gamma, H_0} (\eta) \ge 0 $, $ \forall \ \eta \in W^{2,2} (\Sigma)$ .

\vh

\item[(ii)] $ F_{\gamma, H_0} (\eta) = 0 $ if and only if
$ \eta \in \L (\gamma) $,
where $$ \L(\gamma ) = \lf\{  a_0 + \sum_{i=1}^3 a_i X^i  \ | \ a_0, a_1, a_2, a_3 \ \mathrm{are \ arbitrary \ constants} \ri\}.  $$

\item[(iii)] If $\eta\in \L(\gamma)$, then $Q_{\gamma,H_0}(\eta, \phi )=0$,  $ \forall \ \phi \in W^{2,2}(\Sigma)$.

\end{enumerate}

\end{lemma}

\begin{remark}
(i) and (ii) are proved in  \cite[Corollary 3.1]{MiaoTamXie11}. (iii) is a direct consequence of (i) and (ii) by considering the first variation of $F_{\gamma,H_0}$.
\end{remark}

\begin{remark}
Since any two isometric embeddings of $(\Sigma, \gamma)$ differ by a rigid motion in $\R^3$,
the space $\L(\gamma) $ defined above  is independent on the choice of $X$.
\end{remark}

Lemma \ref{lem-earlier} shows $ F_{\gamma, H} (\cdot) $ vanishes on $ \L (\gamma) $ when  $H = H_0$.
For an arbitrary  $H$, we have the following from (3.12) in \cite[Proposition 3.2]{MiaoTamXie11}.

\begin{lemma}\label{lem-earlier-2}
Suppose  $H>0$. For any $\eta  =a_0+\sum_{i=1}^3a_iX^i \in \mathcal{L}(\gamma)$,
$$
F_{\gamma,H}(\eta)=|a|^2\int_\Sigma(H_0-H)dv_\gamma+\int_\Sigma\la a,\nu_0\ra^2   \frac{(H_0-H)^2}H dv_\gamma,
$$
where $a=(a_1,a_2,a_3)$ and $\nu_0$ is the unit outward normal to $(\Sigma, \gamma)$ when it is  isometrically embedded in $\R^3$.
\end{lemma}

 If $\int_\Sigma(H_0-H)dv_\gamma>0$, then Lemma \ref{lem-earlier-2} implies
$$
F_{\gamma,H}(\eta)\ge \beta\int_\Sigma(\Delta\eta)^2dv_\gamma
$$
for some $\beta>0$ for all $\eta\in \mathcal{L}(\gamma)$.

Next we estimate $ F_{\gamma, H_0} (\eta) $  for  $\eta$ that are $\gamma$-$L^2$ orthogonal to $ \L (\gamma)   $.

\begin{lemma}  \label{lem-supplement-1}
Let $ \sigma$ be a metric of positive Gaussian curvature  on $ \Sigma$.
There exist positive constants  $ {\delta}$ and $ \beta$ such that
if $\gamma$ is a metric on $ \Sigma $ satisfying
$  || \gamma - \sigma ||_{C^{2, \alpha}  (\Sigma, \sigma)} < {\delta} , $
then
$$ F_{\gamma, H_0} (\eta )
\ge \beta \int_\Sigma ( \Delta_\gamma \eta )^2 \ dv_\gamma
$$
for all $ \eta \in W^{2,2} (\Sigma)$ that is $\gamma$-$L^2$ orthogonal to $ \L (\gamma) $.
Here $ \Delta_\gamma$ denotes the Laplacian on $(\Sigma, \gamma)$.
\end{lemma}

\begin{proof}   We argue by contradiction. Suppose it is not true,
then there exists a sequence of  metrics $ \{ \gamma_k \}$ on $ \Sigma$ and a sequence of
functions $ \{ \eta_k \} \subset W^{2,2} (\Sigma) $ such that
\be \label{eq-metric-close}
\lim_{k \rightarrow \infty} || \gamma_k - \sigma ||_{C^{2, \alpha} (\Sigma , \sigma)  } = 0 ,
\ee
\be  \label{eq-orthogonality}
\int_\Sigma \eta_k \phi \ d v_{\gamma_k} = 0 , \ \forall  \phi \in \L(\gamma_k), \ k=1,2, 3, \ldots  ,
\ee
and
\be \label{eq-FHG-1}
\begin{split}
F_{\gamma_k, H_0^k }(\eta_k)
 \le & \ \frac{1}{k}  \int_\Sigma ( \Delta_{k}  \eta_k )^2 \ dv_{\gamma_k}    .
\end{split}
\ee
Here $ H^k_0$ is the $H_0$ associated to $ \gamma_k$ and $ \Delta_k$ stands for  $ \Delta_{\gamma_k} $.
We  normalize $ \eta_k $ such that
\be \label{eq-L2-bd}
\int_\Sigma \eta_k^2 \ dv_{\gamma_k} = 1 .
\ee

Let $ X: (\Sigma, \sigma) \rightarrow \R^3$ be an isometric embedding of $(\Sigma, \sigma)$.
By \eqref{eq-metric-close} and the result of Nirenberg  \cite[p.353]{Nirenberg}, for each large $k$,
 there exists an isometric embedding $ X^k$ of $ (\Sigma, \gamma_k)$ in $ \R^3$
such that
\be \label{eq-embedding-close}
|| X^k - X ||_{C^{2, \alpha} (\Sigma, \sigma) } \le C || \gamma_k - \sigma ||_{C^{2, \alpha} (\Sigma, \sigma) }
\ee
where $ C$ is some constant depending only on $ \sigma$.
Let  $\Pi^k_0 $, $ \Pi_0$  be the second fundamental form  of  $X^k(\Sigma) $, $X(\Sigma)$ respectively.
\eqref{eq-embedding-close} implies that  $\Pi^{k}_0 $ and $\Pi_0$
(viewed as $(0,2)$ tensor fields  on $\Sigma$ through the pullback by $X^k$ and $X$) satisfy
\be \label{eq-Pi-close}
\lim_{k \rightarrow \infty} || \Pi_0^k - \Pi_0 ||_{C^{0, \alpha} (\Sigma, \sigma) } = 0 .
\ee
Consequently, $ \{ H_0^k \}$ converges to $ H_0$ uniformly on $\Sigma$, where
 $H_0$ is the mean curvature of   $X(\Sigma)$.

It follows from \eqref{eq-metric-close} - \eqref{eq-L2-bd},  \eqref{eq-Pi-close},  the interpolation inequality for Sobolev spaces and the $L^2$-estimates that
\be \label{eq-FHG-2}
\begin{split}
\int_\Sigma\frac{ ( \Delta_{k}  \eta_k )^2 }{ H_0^k  } \ d v_{\gamma_k} \le & \  \int_\Sigma  \Pi_0^k ( \nabla_{k}  \eta_k  ,\nabla_{k}  \eta_k  )  \  dv_{\gamma_k}
 +  \frac{1}{k}  \int_\Sigma ( \Delta_{k}  \eta_k )^2 \ dv_{\gamma_k}    \\
 \le & \ \frac12  \int_\Sigma\frac{ ( \Delta_{k}  \eta_k )^2 }{ H_0^k  } \ d v_{\gamma_k}  + C_1
 +  \frac{1}{k}  \int_\Sigma ( \Delta_{k}  \eta_k )^2 \ dv_{\gamma_k}
\end{split}
\ee
where  $\nabla_k$ is  the gradient on $(\Sigma, \gamma_k)$. Here and below,
 $ \{ C_i  \} $ always  denote  positive constants that are independent on $ k$.
Now \eqref{eq-FHG-2} shows
\be \label{eq-lap-k-bd}
\int_\Sigma ( \Delta_{\gamma_k} \eta_k )^2 \ d v_{\gamma_k} \le C_2 ,
\ee
which combined with  \eqref{eq-metric-close}, \eqref{eq-L2-bd} and the $L^2$-estimates implies
\be \label{eq-w22-k}
|| \eta_k ||_{W^{2,2}(\Sigma, \gamma_k)} \le C_3.
\ee
By  \eqref{eq-metric-close}, this in turn shows
 % \be \label{eq-w22}
 $  || \eta_k ||_{W^{2,2}(\Sigma, \sigma)} \le C_4 $.
%\ee
Hence  $ \exists $  $ \eta \in W^{2,2} (\Sigma)$ such that, passing to a subsequence,
$\{ \eta_k \}$ converges to $\eta$ weakly in $W^{2,2}(\Sigma, \sigma)$ and strong in $W^{1,2} (\Sigma, \sigma)$.
By  \eqref{eq-metric-close} - \eqref{eq-orthogonality} and \eqref{eq-L2-bd} - \eqref{eq-embedding-close},
$\eta$ is $\sigma$-$L^2$ orthogonal to $\L(\sigma)$ with
%\be \label{eq-eta-1-in-L}
$ \int_\Sigma \eta^2 dv_{\sigma} = 1 $.
%\ee

We claim
\be \label{eq-F-eta-be-0}
F_{\gamma,H_0}(\eta)=  \int_\Sigma \frac{ ( \Delta \eta )^2 }{ H_0} - \Pi_0 ( \nabla \eta , \nabla \eta) \ d v_\sigma\le0 .
\ee
If this is true, then we have a contradiction by   Lemma \ref{lem-earlier}.

To prove \eqref{eq-F-eta-be-0},
%since $\Delta\eta_k$ only weakly converge to $\Delta\eta$ in $L^2$,
we apply Lemma \ref{lem-earlier} to obtain  $F_{\gamma_k,H_0^k}(\eta_k-\eta)\ge0$.
By \eqref{eq-metric-close}, \eqref{eq-FHG-1} \eqref{eq-embedding-close}, \eqref{eq-Pi-close} and \eqref{eq-lap-k-bd}, we have
\be \label{eq-before-limit}
\begin{split}
\frac{C_2}k \ge&  \ F_{\gamma_k,H_0^k}(\eta_k)\\
\ge & \ 2Q_{\gamma_k,H_0^k}(\eta_k,\eta)-F_{\gamma_k,H_0^k}(\eta)\\
= & \ 2Q_{\gamma ,H_0 }(\eta_k,\eta)-F_{\gamma,H_0 }(\eta)+o(1)
\end{split}
\ee
as $k\to\infty$. Using the fact that $\eta_k$ converge to $\eta$ weakly in $W^{2,2}(\Sigma)$
and strongly in $W^{1,2}(\Sigma)$, we conclude from \eqref{eq-before-limit} that
\bee
\begin{split}
0\ge& \ 2Q_{\gamma ,H_0 }(\eta,\eta)-F_{\gamma,H_0 }(\eta)\\
=& \ F_{\gamma,H_0}(\eta)
\end{split}
\eee
which verifies \eqref{eq-F-eta-be-0}. The Lemma is proved.
\end{proof}

Often  we need  estimates of    $\int_\Sigma |\nabla \eta |^2 d v_\gamma $ by
$\int_\Sigma (\Delta \eta )^2 d v_\gamma $ which depend explicitly on the metric $ \gamma$.
This can be given in terms of   eigenvalues of $(\Sigma, \gamma)$.

 \begin{lemma}\label{lma-Laplace-gradient}
 Let $(M, g)$ be a compact Riemannian manifold without boundary.
 Let $ E_k $ be the space of eigenfunctions with the $k$-th nonzero eigenvalue $\mu_k $ of $(M, g)$.
 Let   $E_0 $ be the space of constant functions.
 Suppose $ \phi \in W^{2,2} (M)$ is $g$-$L^2$ orthogonal to $ E_0, E_1, \ldots, E_{k-1} $,
 then
        \bee
       \mu_k\int_{M}|\nabla  \phi|^2dv_g \le \int_{M} (\Delta \phi)^2 dv_g.
       \eee
        In particular,
        \bee
       \mu_1\int_{M }|\nabla  \phi|^2dv_g  \le \int_{M } ( \Delta \phi)^2 dv_g, \ \forall \ \phi \in W^{2,2}(M) .
       \eee
       \end{lemma}
\begin{proof}
Since
$$
\int_{M}|\nabla \phi|^2 dv_g =-\int_{\Sigma}\phi\Delta\phi dv_g \le  \lf(\int_{M}\phi^2dv_g \ri)^\frac12\lf(\int_{\Sigma}(\Delta\phi)^2
dv_g \ri)^\frac12,
$$
we have
$$
\frac{ \int_M | \nabla \phi |^2 d v_g }{ \int_M \phi^2 d v_g } \le \frac{ \int_M ( \Delta \phi )^2 d v_g }{ \int_M | \nabla \phi |^2 d v_g } ,
$$
from which the claim follows.
\end{proof}

\section{Sufficient conditions for the positivity of $F_{\gamma, H}$} \label{sect: sufficient}

 In this section, we provide some sufficient conditions guaranteeing
 the positivity of $F_{\gamma, H}$.
 Note that Lemma \ref{lem-supplement-1} implies
 $$  \inf_{ \eta \in  \L(\gamma)^\perp }    \frac{ F_{\gamma, H_0}  (\eta ) }{ \int_\Sigma ( \Delta \eta )^2 d v_\gamma }
  > 0  $$
 where $ \L(\gamma)^\perp $ is the space of functions that are $ \gamma$-$L^2$ orthogonal to $ \L (\gamma)$.

\begin{prop} \label{prop-perturbation-big-BY}
Let $ \gamma$ be a metric of positive Gaussian curvature on $ \Sigma$.
Let $\beta$ be a positive constant such that
\be \label{eq-F-beta}
F_{\gamma,H_0}(\eta)\ge \beta\int_\Sigma (\Delta\eta)^2dv_\gamma, \ \forall \ \eta \in \L(\gamma)^\perp.
\ee
Suppose the first nonzero eigenvalue of $(\Sigma,\gamma)$ is at least $\lambda>0$. Then $ \exists $
$\delta>0$, depending only on $ \beta$, $\lambda,$ $ H_0$ and a given constant $ \a > 0 $, such that if $H\ge\a $ is a  function on $\Sigma$
satisfying
\begin{enumerate}

\item[(a)]  $ \displaystyle  \int_\Sigma (H_0 - H ) \ d v_\gamma  > 0 $

\item[(b)] $ \displaystyle \sup_\Sigma \left|( H_0 - H)_- \right|  < \delta  $  and
$ \displaystyle
  \frac{ \int_\Sigma |H_0 - H|^2 \ d v_\gamma }{ \int_\Sigma (H_0 - H ) \ d v_\gamma }  < \delta       $,
\end{enumerate}
where $(H_0-H)_{-} =\min\{H_0-H,0\}$,
then
$ \displaystyle
F_{\gamma, H} (\eta) \ge \tilde{\beta} \int_\Sigma ( \Delta \eta)^2 \ dv_\gamma
$,  $ \forall $  $\eta \in W^{2,2} (\Sigma) $, where  $\tilde \beta>0$ is a constant  depending  only on  $H_0$ and $\beta$.
\end{prop}

\begin{proof} Given any constant $\alpha > 0 $, let  $H \ge \a $ be  a function on $ \Sigma$ satisfying (a) and (b) with
$\delta>0$ to be chosen later.

Let $ X = (X_1, X_2, X_3)$ be an isometric embedding of $(\Sigma, \gamma)$ in $ \R^3$.
Given any $\eta \in W^{2,2}(\Sigma)$, decompose
$
\eta = \eta_1 + \eta_2
$
where $ \eta_1 = a_0 + \sumi a_i X_i  \in \L  (\gamma) $ and $ \eta_2 \in  \L(\gamma)^\perp$. Let  $a  =(a_1, a_2, a_3 )  $.
If $a=0$, then
\be\label{eq-a=0}
\begin{split}
F_{\gamma, H} (\eta) =& \ F_{\gamma,H_0}(\eta_2)+\int_\Sigma(H_0-H)\lf(\frac{(\Delta \eta_2)^2}{HH_0}+|\nabla \eta_2 |^2\ri)dv_\gamma\\
\ge& \ \beta \int_\Sigma (\Delta \eta_2 )^2dv_\gamma-\delta\lf(\frac1{\a\inf_{\Sigma}H_0}+\frac1\lambda\ri) \int_\Sigma(\Delta \eta_2)^2dv_\gamma\\
=& \ \lf[\beta-\delta\lf(\frac1{\a\inf_{\Sigma}H_0}+\frac1\lambda\ri)\ri]\int_\Sigma (\Delta\eta )^2dv_\gamma
\end{split}
\ee
where we have used \eqref{eq-F-beta}  and Lemma \ref{lma-Laplace-gradient}.

Now suppose $a\neq0$, we may normalize $a$ so that $|a|=1$. Then  for any $\e_1, \e_2>0$,
\be\label{eq-anot=0}
\begin{split}
& \ F_{\gamma, H} (\eta)-F_{\gamma,H_0}(\eta)\\
 = & \ \int_\Sigma\lf[\lf(\frac1H-\frac1{H_0}\ri) (\Delta\eta)^2+(H_0-H)|\nabla \eta|^2\ri]dv_\gamma\\
= & \ \int_\Sigma\lf[\lf(\frac1H-\frac1{H_0}\ri) (\Delta\eta_1)^2+(H_0-H)|\nabla \eta_1|^2\ri]dv_\gamma\\
& \ +\int_\Sigma\lf[\lf(\frac1H-\frac1{H_0}\ri) (\Delta\eta_2)^2+(H_0-H)|\nabla \eta_2|^2\ri]dv_\gamma\\
& \ +2\int_\Sigma\lf[\lf(\frac1H-\frac1{H_0}\ri)\Delta\eta_1\cdot\Delta\eta_2
+(H_0-H)\la\nabla\eta_1,\nabla\eta_2\ra \ri]dv_\gamma\\
\ge & \ \int_\Sigma (H_0-H)dv_\gamma-\delta\lf(\frac1{\a\inf_{\Sigma}H_0}+\frac1\lambda\ri) \int_\Sigma (\Delta\eta_2)^2dv_\gamma\\
& \ - \e_1^{-1}\int_\Sigma\lf(\frac1H-\frac1{H_0}\ri)^2(\Delta\eta_1)^2dv_\gamma
-\e_1\int_\Sigma(\Delta\eta_2)^2dv_\gamma\\
& \ -\e_2^{-1}\int_\Sigma\lf(H-H_0\ri)^2|\nabla\eta_1|^2dv_\gamma
-\e_2\int_\Sigma|\nabla\eta_2|^2dv_\gamma\\
\ge &
\ \lf( 1-\e_1^{-1} \a^{-2} \delta -  \e_2^{-1} \delta \ri) \int_\Sigma(H_0-H)dv_\gamma
\\
&-\lf(  \delta\lf(\frac1{\a\inf_{\Sigma}H_0}+\frac1\lambda\ri)+\e_1+\lambda^{-1}\e_2\ri)  \int_\Sigma(\Delta\eta_2)^2dv_\gamma
\end{split}
\ee
%%%%%%%%%%%%%%%%%%%%
% \fixme{changes occur in line - 2 in (3.3)}
%%%%%%%%%%%%%%%%%%%%
where we have used Lemma \ref{lem-earlier-2},  the fact $(\Delta\eta_1)^2= \la a,\nu_0\ra^2 H_0 ^2$ and $|\nabla\eta_1|^2=1-\la a,\nu_0\ra^2$. On the other hand, it follows from Lemma  \ref{lem-earlier} and \eqref{eq-F-beta} that
\bee
F_{\gamma,H_0}(\eta)=
F_{\gamma,H_0}(\eta_2)\ge\beta\int_\Sigma(\Delta\eta_2)^2dv_\gamma.
\eee
Hence
\bee
\begin{split}
 F_{\gamma, H} (\eta)\ge&\lf( 1-\e_1^{-1} \a^{-2} \delta -  \e_2^{-1} \delta \ri) \int_\Sigma(H_0-H)dv_\gamma
\\
& + \lf[\beta-\lf(\delta\lf(\frac1{\a\inf_{\Sigma}H_0}+\frac1\lambda\ri)+\e_1+\lambda^{-1}\e_2\ri)\ri]\int_\Sigma(\Delta\eta_2)^2dv_\gamma .
\end{split}
\eee
Let  $\e_1=\lambda^{-1}\e_2=\frac\beta4$ and choose  $\delta>0$  such  that
$ 1- ( \e_1^{-1} \a^{-2} +  \e_2^{-1} ) \delta  \ge \frac12$, and $\delta\lf(\frac1{\a\inf_{\Sigma}H_0}+\frac1\lambda\ri)\le \frac\beta4$, then
$$
F_{\gamma,H }(\eta)\ge \frac12\int_\Sigma(H_0-H)dv_\gamma+\frac\beta 4\int_\Sigma(\Delta\eta_2)^2dv_\gamma.
$$
Combining this with \eqref{eq-a=0} and the fact $(\Delta\eta_1)^2\le H_0^2$, we conclude that the Proposition is true.
\end{proof}

In  \cite[Theorem 3.1]{MiaoTamXie11}, it was proved that $ F_{\gamma, H} $ is positive definite if
$ H \le H_0 $ and $ \int_\Sigma (H_0 - H) d v_\gamma > 0 $.
By  arguments similar to the proof of Proposition \ref{prop-perturbation-big-BY},
it can be shown  that $ F_{\gamma, H}$ remains positive definite if $ H $ is allowed to be slightly bigger than $ H_0$.
First, we give a quantitative estimate of the case $ H \le H_0$.

\begin{lemma}\label{lem: negative-part}
Let $ \gamma $ be a metric of positive Gaussian curvature on $ \Sigma$.
Let $\beta>0$ be a constant such that
$$
F_{\gamma,H_0}(\eta)\ge \beta\int_\Sigma (\Delta\eta)^2dv_\gamma, \ \forall \ \eta \in \L(\gamma)^\perp .
$$
Let $\lambda>0$ be a lower bound for the first nonzero eigenvalue of $(\Sigma, \gamma)$.
Given any positive function  $ H $  on $ \Sigma$ with  $ H \le H_0$,  let  $ \alpha > 0 $ be a lower bound of $ H$.
Then
\bee
F_{\gamma,H}(\eta_1 + \eta_2 ) \ge \frac{ \frac\beta2  }{ \a^{-1}+\lambda^{-1}\sup_\Sigma H_0
+\frac\beta2 }\int_\Sigma(H_0-H) dv_\gamma+\frac\beta2\int_\Sigma(\Delta\eta_2)^2dv_\gamma
\eee
for any $ \eta_2 \in \L(\gamma)^\perp $ and  $ \eta_1=a_0+\sum_{i=1}^3 a_iX^i \in \L(\gamma)$ with
$a = ( a_1, a_2, a_3) $  being  a unit vector.
%%%%%
% \fixme{The case $ a = 0 $ was dropped in the statement since the estimate seems not include the $a=0$ case.}
%%%%%
\end{lemma}

\begin{proof}
%It follows from $ H \le H_0 $ and \eqref{eq-F-beta-2} that
%\be\label{eq-neg-1}
%F_{\gamma,H}(\eta) \ge F_{\gamma, H_0}(\eta) = F_{\gamma,H_0}(\eta_2)\ge \beta\int_\Sigma(\Delta\eta_2)^2.
%\ee
% Suppose  $ | a | = 1 $.
 Similar to \eqref{eq-anot=0}, using the fact  $H_0\ge H$, we have for any constant $ 0 < \epsilon < 1 $ that
\bee
\begin{split}
&F_{\gamma, H} (\eta)-F_{\gamma,H_0}(\eta)\\
\ge&(1-\e)\int_\Sigma\lf[\lf(\frac1H-\frac1{H_0}\ri) (\Delta\eta_1)^2+(H_0-H)|\nabla \eta_1|^2\ri]dv_\gamma\\
&+(1-\e^{-1})\int_\Sigma\lf[\lf(\frac1H-\frac1{H_0}\ri) (\Delta\eta_2)^2+(H_0-H)|\nabla \eta_2|^2\ri]dv_\gamma\\
\ge & \ (1-\e)\int_\Sigma (H_0-H)dv_\gamma+(1-\e^{-1})\lf(\a^{-1}+\lambda^{-1}\sup_\Sigma H_0\ri)\int_\Sigma(\Delta\eta_2)^2dv_\gamma .
\end{split}
\eee
Hence
\be\label{eq-neg-2}
\begin{split}
 F_{\gamma, H} (\eta)\ge & \ (1-\e)\int_\Sigma (H_0-H)dv_\gamma\\
 & \ +\lf[\beta+(1-\e^{-1})\lf(\a^{-1}+\lambda^{-1}\sup_\Sigma H_0\ri)\ri]\int_\Sigma(\Delta\eta_2)^2dv_\gamma
 \end{split}
 \ee
 where we can  choose $ 0 < \epsilon < 1 $ such that
 $$
 \beta+(1-\e^{-1})\lf(\a^{-1}+\lambda^{-1}\sup_\Sigma H_0\ri)=\frac\beta2 .
 $$
 The Lemma now follows from  \eqref{eq-neg-2}.
\end{proof}

\begin{prop}  \label{prop: negative-part}
Let $ \gamma $ be a metric of positive Gaussian curvature on $ \Sigma$. Let $\beta>0$ be a constant such that
$$
F_{\gamma,H_0}(\eta)\ge \beta\int_\Sigma (\Delta\eta)^2dv_\gamma, \ \forall \ \eta \in \L(\gamma)^\perp .
$$
Let $\lambda>0$ be a lower bound for the first nonzero eigenvalue of $(\Sigma, \gamma)$.
Given a positive function $H$ on $ \Sigma$, let $ (H_0 - H)_{-}=\min\{H_0-H,0\} $.
Let $ \alpha > 0 $ be any  lower bound of $ H$.  Define
$\a_1=\min\{\a,\inf_\Sigma H_0\}$,
$$
\theta=\frac{ \frac\beta2 }{ \a_1^{-1}+\lambda^{-1}\sup_\Sigma H_0 +  \frac\beta2}
\ \ \mathrm{and} \ \
\delta=\frac\beta4\lf(\frac1{\a\inf_\Sigma H_0}+\frac1\lambda\ri)^{-1}.
$$
Suppose
\begin{enumerate}
\item[(i)] $ \displaystyle  \theta \int_\Sigma (H_0 - H) + 2 \int_\Sigma (H_0 - H)_- > 0  $
\item[(ii)]  $ \displaystyle
\sup_\Sigma \lf| { ( H_0 - H)_ - }  \ri| < \delta, $
\end{enumerate}
then
$ \displaystyle  F_{\gamma, H} (\eta) \ge \tilde{ \beta} \int_\Sigma ( \Delta \eta)^2 d v_\Sigma  $,
 $ \forall \ \eta \in W^{2,2} (\Sigma)$, for some $ \tilde{\beta} > 0$.
\end{prop}

\begin{proof}
Suppose $ H$ satisfies (i) and (ii). Given $ \eta \in W^{2,2}(\Sigma)$,
let $\eta=\eta_1+\eta_2$  and $\eta_1=a_0+\sum_{i=1}^3 a_iX^i$ be given as in the proof of Proposition
\ref{prop-perturbation-big-BY}.
Let $a=(a_1,a_2,a_3)$. If $ a =0$, similar to \eqref{eq-a=0},  we have
\be\label{eq-neg-3}
\begin{split}
F_{\gamma,H}(\eta) \ge&  \ \beta\int_\Sigma(\Delta\eta_2)^2+\int_\Sigma(H_0-H)\lf(\frac{(\Delta\eta_2)^2}{H_0H}+|\nabla\eta|^2\ri)dv_\gamma\\
\ge & \ \lf[\beta-\delta\lf(\frac1{\a \inf_\Sigma H_0}+\frac1\lambda\ri)\ri]\int_\Sigma(\Delta\eta_2)^2.
\end{split}
\ee
Next, suppose  $a\neq0$. We normalize $a$ so that $|a|=1$. Define
%%%%%%%
% \fixme{definition of $H_1$, $H_2$ simplified.}
%%%%%%%%
\bee
H_1= \min \{ H_0, H \} \ \mathrm{and} \
H_2 = \max \{H_0, H\} .
\eee
Then  $H_1+H_2=H+H_0$, $1/H_1+1/H_2=1/H+1/H_0$ and $H_0-H_2=(H_0-H)_-$. By Lemma \ref{lem: negative-part},
we have
\be \label{prop-3-a-notzero}
\begin{split}
F_{\gamma,H}(\eta)
=& \ F_{\gamma,H_1}(\eta)
+\int_\Sigma\lf(H_0-H_2\ri)\lf(\frac{(\Delta\eta)^2}{H_2H_0}
+|\nabla\eta|^2\ri)dv_\gamma\\
\ge& \ \theta \int_\Sigma(H_0-H_1)dv_\gamma+\frac\beta2\int_\Sigma(\Delta\eta_2)^2dv_\gamma\\
& \ +2\int_\Sigma\lf(H_0-H\ri)_-\lf(\frac{(\Delta\eta_1)^2}{H_2H_0}
+|\nabla\eta_1|^2\ri)dv_\gamma \\
& \ +2\int_\Sigma\lf(H_0-H\ri)_-\lf(\frac{(\Delta\eta_2)^2}{H_2H_0}
+|\nabla\eta_2|^2\ri)dv_\gamma\\
\ge & \ \theta \int_\Sigma(H_0-H )dv_\gamma   +2\int_\Sigma\lf(H_0-H\ri)_- dv_\gamma  \\
& \ + \lf[\frac\beta2- 2 \delta\lf( \frac1{ \a \inf_\Sigma H_0} +\frac1\lambda\ri)\ri]\int_\Sigma(\Delta\eta_2)^2dv_\gamma .
\end{split}
\ee
Proposition \ref{prop: negative-part} now follows from \eqref{eq-neg-3} and  \eqref{prop-3-a-notzero}.
\end{proof}

\section{Nearly round surfaces in AF manifolds}  \label{sect: application}

In this section, we apply Lemma \ref{lem-supplement-1} and Proposition \ref{prop-perturbation-big-BY}  to study
the positivity of $ F_{\gamma, H} $ on certain ``large surfaces" near infinity in an asymptotically flat $3$-manifold,
 based on existing results in \cite{FanShiTam07, ShiWangWu09}.

We adopt the following definition in \cite{FanShiTam07} for an asymptotically flat $3$-manifold and an admissible
coordinate chart.
\begin{df}
A Riemannian $3$-manifold is called asymptotically flat (AF) of order $\tau$ (with one end) if there is a compact
set $K$ such that $ M \setminus K$ is diffeomorphic to $ \R^3 \setminus B_R(0)$, where $ B_R(0)$ is a coordinate
ball of radius $R>0$ centered at the origin, such that in the standard coordinates $\{ z_i \}$ on $ \R^3$ the metric $g$ satisfies
\be \label{eq-AF-metric-1}
g_{ij} = \delta_{ij} + h_{ij}
\ee
with
\be \label{eq-AF-metric-2}
| h_{ij} | + |z| | \p h_{ij} | + |z|^2 | \p^2 h_{ij} | + | z|^3 | \p^3 h_{ij} |  = O ( | z |^{-\tau} )
\ee
for some constant $ \tau > \frac12 $. Here $|z|$ and $\p $ denote the coordinate length of $z$ and the usual partial
derivative operator on $ \R^3$ respectively.

A coordinate chart  $\{ z_i \}$ on $M$  in which the metric $g$ satisfies the above
conditions  \eqref{eq-AF-metric-1}-\eqref{eq-AF-metric-2} is called  an  admissible coordinate chart.

\end{df}

Large coordinate spheres in an admissible coordinate chart are examples of {\em nearly round surfaces}
(see \cite{ShiWangWu09}).  These surfaces are  intrinsically defined as follows.

\begin{df}[\cite{ShiWangWu09}]
On an asymptotically flat $3$-manifold $(M, g)$ of order $ \tau > \frac12$,
let $ r (x) $ be the $g$-distance from $x$ to a fixed point.
A $1$-parameter family of surfaces
$\{\Sigma_r \}$, where $ r = \min_{\Sigma_r} r(x) $ and
 $\Sigma_r $  is topologically a $2$-sphere, is called nearly round
as $ r  $ tends to infinity    if
\begin{enumerate}
\item $| \overset\circ{A} |+r |\nabla \overset\circ A| \leq C r^{-1-\tau}$

\item $ \max_{x \in \Sigma_r} r(x) \le C r $

\item $ \mathrm{diam}(\Sigma_r) \le C r $

\item $ \mathrm{Area} (\Sigma_r) \le C r^2 $ .

\end{enumerate}
Here $\overset\circ A$ is the traceless part of  the second fundamental form of $\Sigma_r$,
 $\nabla$ and $|\cdot|$  denote  the covariant
derivative and  the norm on  $ \Sigma_r$  with respect to the induced metric,
$ \mathrm{diam}(\cdot) $ and $ \mathrm{Area}(\cdot) $ denote the diameter and the area of a surface, and
   $C$  is a constant independent of $r$.
\end{df}

As shown in \cite{ShiWangWu09},  other  examples of nearly round surfaces include  the  constant mean curvature
surfaces  constructed in \cite{HuiskenYau96} and \cite{Ye96}.  The main result in this section is

\begin{thm} \label{thm-nearly-round-surface}
Let $(M, g)$ be an asymptotically flat $3$-manifold of
order $\tau > \frac12 $.
Let $ \{ \Sigma_r \}$ be a family of nearly round surfaces as $r$ tends to infinity.
Suppose the ADM mass  of $(M, g)$  is positive.
 Then there exist constants $ R > 0 $ and $C>0$ such that
\bee \label{eq-nearly-round-surface-1}
 \int_{\S_r} \lf[ \frac{ ( \Delta \eta )^2 }{ H  } + ( H_0 - H ) | \nabla \eta  |^2 - \Pi_0 ( \nabla \eta ,\nabla \eta)\ri]dv_{r}
 > C r  \int_{S_r} (\Delta \eta)^2 d v_r ,
\eee
$\forall $ $\eta \in W^{2,2} (\S_r) $ and $ \forall \ r > R$.

As a result, the Brown-York mass of $\S_r$ is a strict local minimum of the Wang-Yau quasi-local energy of $\S_r$ for sufficiently
large $r$.
\end{thm}

\begin{proof}
Let $ K $ be a compact set such that $ M \setminus K$ carries an admissible coordinate chart.
Let $ \hat{g} $ be a background  Euclidean metric on $ M \setminus K$.
Let $H_r$ and $ \hat{H}_r$ be the mean curvature of $\Sigma_r $ in
$(M \setminus K, g)$ and $ (M \setminus K, \hat{g})$ respectively.
By (2.13) in \cite{ShiWangWu09}, one has
\be \label{eq-nearly-round-h}
H_r = \hat{H}_r + O (r^{-1 -\tau}).
\ee
By Proposition 3.2 and Theorem 4 in  \cite{ShiWangWu09}, for each sufficiently large $r$,
there exists a number $r_0 \in \R$  such that
\be \label{eq-r0-and-r}
 C^{-1} r \le r_0 \le C r ,
 \ee
 \be
 \hat{H}_r= \frac{2}{r_0} + O(r^{-1 - \tau} ),
 \ee
 \be \label{eq-nearly-round-h-more}
| H_{0_{r}}  - \frac{2}{r_0} |  \le C r_0^{-1 - \tau},
\ee
and
 \be \label{eq-K-1}
K_r = \frac{1}{r_0^2} + O (r^{-2 -\tau}), \
 | \nabla K | = O ( r^{-3 - \tau} ).
 \ee
Here $C$ is a constant independent on $r$,
$K_r$ is the Gaussian curvature of $  \gamma_r $, where  $\gamma_r$ is the induced metric on $ \Sigma_r $  from $g$,
and  $H_{0_ r}$ is the mean curvature of $ (\Sigma_r, \gamma_r)  $ when it is isometrically embedded in $\R^3$.

Let $ \sigma_r  = r_0^{-2} \gamma_r  $. It follows from \eqref{eq-r0-and-r},  \eqref{eq-K-1} and the proof of
Theorem 3 in \cite{ShiWangWu09} (in particular (3.1) in \cite{ShiWangWu09}) that,
for each large $r$,  there is a conformal map $\Phi_r$ from $(\mathbb{S}^2, \sigma_0)$
to $ (\Sigma_r, \sigma_r)  $ such that
%%%%%%%%%%%%%%%
% \fixme{$\gamma_r$ in (4.8) replaced by $ \sigma_r$}
%%%%%%%%%%%%%%
\be \label{eq-metric-close-nearly-round}
||\Phi_r^*(\sigma_r) - \sigma_0||_{C^{2,\a} (\mathbb{S}^2) }=O(r^{-\tau}) .
\ee
Here $\sigma_0$ is the standard metric on $\mathbb{S}^2$.

Let $H(r)$ be the mean curvature of  $ \Sigma_r $ in $ (M, r^{-2}_0 g )$
 and  $H_0 (r)$ be the mean curvature of $(\Sigma_r, \sigma_r)$
 when it is isometrically embedded in $ \R^3$.
%Denote  the Gaussian curvature of  $ \sigma_r  = r_0^{-2} \gamma_r  $ by $K(r)$.
It follows from \eqref{eq-nearly-round-h} --  \eqref{eq-nearly-round-h-more}  that
 \be \label{eq-H-condition-nearly-round}
 H_0 (r) - H(r) = O (r^{-\tau} ).
 \ee
 %and
 % \be \label{eq-K-condition-nearly-round}
% || K(r) - 1 ||_{C^{1} (\Sigma_r, \sigma_r )} =  O(r^{-\tau} ) .
% \ee
On the other hand,
by Theorem 5 in \cite{ShiWangWu09},
%see also \cite{FanShiTam07},
 the Brown-York mass of $\Sigma_r $ in $(M, g)$ satisfies
\be \label{eq-ADM-nearly-round-surface}
\lim_{r \rightarrow \infty} \int_{\Sigma_r } ( H_{0_r} - H_r  ) d v_{\gamma_r}
= 8 \pi  \m  (g)
\ee
where $\m (g) > 0 $ is the ADM mass of $(M,g)$.
This together with \eqref{eq-r0-and-r}  implies  $ \exists \   R_0 > 0 $ such that
\be \label{eq-big-by-mass}
\int_{\Sigma_r}  (H_0(r) - H(r) )   d v_{\sigma_r}  > C^{-1}  r^{-1}  \pi \m (g), \
\forall \  r > R_0 .
\ee

Now choose $ \gamma = \sigma_r $ and $ H = H (r) $  in Proposition \ref{prop-perturbation-big-BY}.
By  \eqref{eq-metric-close-nearly-round} and  Lemma \ref{lem-supplement-1}, the  constant
$\beta$ in \eqref{eq-F-beta}  and the lower bound  $\lambda$ for the first nonzero eigenvalue
 can both be chosen to be independent on $r$.
Moreover, the conditions (a) and (b) are satisfied for large $r$ by
\eqref{eq-H-condition-nearly-round}, \eqref{eq-big-by-mass} and the fact $ \tau > \frac12$ and $ \m (g) > 0 $.
Therefore,  by Proposition \ref{prop-perturbation-big-BY},  we conclude that $\exists \  R>0$ and $\tilde{\beta} > 0$ such that
\be \label{eq-positivity-sigma-N}
F_{\sigma_r, H (r)} ( \eta ) \ge \tilde{\beta} \int_{\Sigma_r} (\Delta \eta )^2 d v_{\sigma_r}, \ \forall \ \eta \in  W^{2,2} (\Sigma_r), \
\forall \ r > R .
\ee
Theorem \ref{thm-nearly-round-surface} now follows from  \eqref{eq-r0-and-r},  \eqref{eq-positivity-sigma-N},
 and  Theorem \ref{thm-local-minimum}.
\end{proof}

 \section{Small geodesic spheres} \label{sect: small-sphere-2}

Let $(M, g)$ be an arbitrary Riemannian $3$-manifold of nonnegative scalar curvature.
Let $p \in M$ be any given point.   For small $r>0$,
let  $S_r$ be the geodesic sphere of radius $r$ centered at $p$.
Let  $ \gamma_r $ be the induced metric  on $S_r $ and $H(r)$ be the
mean curvature of $ S_r$ in $(M, g)$.
Let $H_{0}(r)$ be the mean curvature
of $(S_r, \gamma_r)$ when it is isometrically embedded in $\R^3$.
By \cite[Theorem 3.1]{FanShiTam07}, the Brown-York mass of $S_r$ in $(M, g)$ satisfies
\be   \label{eq-BY-small-sphere}
\begin{split}
& \ \frac{1}{8\pi} \int_{S_r} (H_{0}(r)  -  H(r) ) d v_{\gamma_r} \\
= & \  \frac{r^3}{12} R(p) + \frac{r^5}{1440} \lf[ 24 | \Ric (p) |^2 - 13 R(p)^2 + 12 \Delta R (p) \ri]
 + O(r^6)
\end{split}
\ee
where $R$ and $ \Ric $ denote the scalar curvature and the Ricci curvature of $(M, g)$ respectively,
and $\Delta$ is the Laplacian on $(M, g)$.
It follows   from \eqref{eq-BY-small-sphere} that  the  condition
\be \label{eq-small-main-1}
\lim_{r\to 0}r^{-5}\int_{S_r}(H_0(r)-H(r))dv_{\gamma_r} >0
\ee
is equivalent to the union of the following three conditions
\begin{enumerate}

\item[(i)]  $R(p)>0$

\item[(ii)] $R(p)=0$ and $|\Ric (p)|^2 >0$

\item[(iii)]  $R(p)=|\Ric(p)| =0$ and $\Delta R(p)>0$.

\end{enumerate}
Note that if $R(p)=0$, then $\Delta R(p)\ge0$ by the assumption $ R \ge 0$.

\begin{thm} \label{thm-small-sphere-main}
Under the above notations,  if the condition \eqref{eq-small-main-1} holds, then
\eqref{eq-intro-F} is true on $ S_r $ for small $r$.
Precisely, we have
\begin{enumerate}

\item[(a)]   If $\mathrm{(i)}$ or $\mathrm{(iii)}$ holds,  then $ \exists$  constants $ r_0 > 0 $ and $C>0$ such that,
$ \forall $ $\eta \in W^{2,2} (S_r) $ and $ \forall \ r < r_0 $,
$$
 \int_{S_r} \lf[ \frac{ ( \Delta \eta )^2 }{ H  } + ( H_0 - H ) | \nabla \eta  |^2 - \Pi_0 ( \nabla \eta ,\nabla \eta)\ri]dv_{r}
 > C r  \int_{S_r} (\Delta \eta)^2 d v_r .
$$

\item[(b)] If  $\mathrm{(ii)}$ holds,  then $ \exists $  constants $ r_0 > 0 $ and $C>0$ such that,
$ \forall $ $\eta \in W^{2,2} (S_r) $ and $ \forall \ r < r_0 $,
$$
 \int_{S_r} \lf[ \frac{ ( \Delta \eta )^2 }{ H  } + ( H_0 - H ) | \nabla \eta  |^2 - \Pi_0 ( \nabla \eta ,\nabla \eta)\ri]dv_{r}
 > C r^5  \int_{S_r} (\Delta \eta)^2 d v_r .
$$
\end{enumerate}
\end{thm}

\begin{proof}[Proof of Theorem \ref{thm-small-sphere-main}(a)]
Let $\{ z_i \} $ be a geodesic normal coordinate chart centered at $p$. For small $r$, $S_r = \{ | z | = r \}$.
The metric $g$ satisfies
\be \label{eq-metric-normal-chart}
g_{ij}  = \delta_{ij} + h_{ij}
\ee
where $ h_{ij} $ is a smooth function near $0$ with $ h_{ij}(0) = 0 $ and $\p h_{ij} (0) = 0 $.
For each fixed $r >0$, define a new coordinate chart $\{ x_i \}$ near $p$  by $ x_i = r^{-1} z_i $.
Let $ \mS^2$ be the unit coordinate sphere in the $x$-space.
We identify $ S_r$ with $ \mathbb{S}^2 $  through the map $ z \mapsto x$.
Let  $ \sigma_r = r^{-2} \gamma_r  $ be the induced metric on $ \mathbb{S}^2$ from $(M, r^{-2} g)$ and
 $H (r) $ be the mean curvature of $\mS^2$   in $(M, r^{-2} g)$. Let
$H_0 (r) $ be the mean curvature of the isometric embedding  of $(\mS^2, \sigma_r)$ in $ \R^3$.

By  \eqref{eq-metric-normal-chart} and the results in \cite{FanShiTam07}(Lemma 3.4 and Theorem 3.1), we have
\be \label{eq-sigma-sigma0-S}
|| \sigma_r - \sigma_0 ||_{C^3 ({\mS^2}, \sigma_0) } = O (r^{2} ) ,
\ee
\be \label{eq-H-r-Ssphere}
H (r) = 2 - \frac{r^2}{3} R_{ij}(p)  {x^i x^j} +  O ( r^3 ),
\ee
\be \label{eq-H-0-Ssphere}
H_0 (r) = 2  +  r^2 \lf( \frac{1}{2} R(p) - \frac43 R_{ij} (p)x^i x^j  \ri)  + O ( r^{ 3} ),
\ee
\be \label{eq-BY-S-sphere}
\begin{split}
& \ \frac{1}{8 \pi} \int_{S_r} (H_0 (r)  - H (r) ) d v_{\sigma_r} \\
= & \  \frac{r^2}{12} R(p) + \frac{r^4}{1440} \lf[ 24 | \Ric (p) |^2 - 13 R(p)^2 + 12 \Delta R (p) \ri] + O(r^5) .
\end{split}
\ee
Here $ \sigma_0$ is the standard metric on $ \mS^2$ and
$R_{ij}(p) = \Ric( \frac{\p}{\p z_i}, \frac{\p }{\p z_j} ) (p) $.

If  $ R(p) > 0 $,  it follows from \eqref{eq-H-r-Ssphere} - \eqref{eq-BY-S-sphere}  that
\be \label{eq-H-differ}
 |H_0(r)-H(r)|=O(r^2)
 \ee
and
\be  \label{eq-BY-S-sphere-Rp}
\int_{\mS^2} (H_0 (r)  - H (r) ) d v_{\sigma_r}   >  \frac{r^2}{3} \pi R(p)
\ee
for small $ r$.
Take $\gamma = \sigma_r$ and $ H = H(r)$ in  Proposition \ref{prop-perturbation-big-BY}.
By \eqref{eq-sigma-sigma0-S}  and  Lemma \ref{lem-supplement-1}, the  constant
$\beta$ in \eqref{eq-F-beta}  and the lower bound  $\lambda$ for the first nonzero eigenvalue
 can both be chosen to be independent on $r$.
Moreover, the conditions (a) and (b) are satisfied for small $r$ by
\eqref{eq-H-differ} and \eqref{eq-BY-S-sphere-Rp}.
Therefore,   Proposition \ref{prop-perturbation-big-BY} implies there exist $ r_0 >0$ and $\tilde{\beta} > 0$ such that
if $ r < r_0$, then
\be \label{eq-positivity-sigma-S-Rp}
F_{\sigma_r, H (r) } ( \eta ) \ge \tilde{\beta} \int_{\mS^2} (\Delta \eta )^2 d v_{\sigma_r},  \ \forall \ \eta \in W^{2,2} (\mS^2).
\ee

If  $R(p)=|\Ric (p)|=0$ and $\Delta R(p)>0$, it follows from \eqref{eq-H-r-Ssphere} - \eqref{eq-BY-S-sphere}  that
\be \label{eq-H-differ-2}
 |H_0(r)-H(r)|=O(r^3)
 \ee
and
\be  \label{eq-BY-S-sphere-Rp-2}
\int_{\mathbb{S}^2}\lf(H_o(r)-H(r)\ri)\dvr\ge  \frac{\Delta R (p)}{240} r^4
\ee
for  small $r$.
Again,   Lemma \ref{lem-supplement-1} and  Proposition \ref{prop-perturbation-big-BY} can be applied to show that
\eqref{eq-positivity-sigma-S-Rp} holds for some $\tilde{\beta}$ and $  r_0$.

Theorem  \ref{thm-small-sphere-main}(a) now  follows  \eqref{eq-positivity-sigma-S-Rp}.
\end{proof}

The case $ R (p) = 0 $ and $ | \Ric(p ) | > 0 $ is more subtle
%The proof of Theorem  \ref{thm-small-sphere-main}(b) is more delicate
because in this case
$
 |H_0(r)-H(r)|=O(r^2)
$
while
$
\int_{\mathbb{S}^2 }\lf(H_o(r)-H(r)\ri)\dvr = O (r^4).
$
The sufficient conditions in Section \ref{sect: sufficient} do not apply in this situation.

To prepare for the proof of Theorem  \ref{thm-small-sphere-main}(b), we choose
$\{ z_i \} $ to be a geodesic normal coordinate chart centered at $p$ such that
 the Ricci curvature of $g$ is diagonalized by $\{ \frac{\p}{\p z_i} \}$  at $p$, i.e.
$$ \Ric \lf( \frac{\p}{\p z_i} , \frac{\p }{\p z_j} \ri)(p)  = \delta_{ij} \lambda_i  $$
where $ \{ \lambda_i \}$  are the eigenvalues of $\Ric(p)$.
Then  $\{ \lambda_i \}$ satisfy
\be
\lambda_1 + \lambda_2 + \lambda_3 = R(p) = 0
\ee
and
\be \label{eq-Ric-lambda}
 \sumi \lambda_i^2 =  | \Ric(p)|^2   >   0 .
\ee
Let $ \{x_i \}$, $\sigma_r$, $\sigma_0$, $H(r)$, $H_0 (r)$ be defined as in the proof of
Theorem  \ref{thm-small-sphere-main}(a). For convenience,
we record  \eqref{eq-sigma-sigma0-S} -- \eqref{eq-BY-S-sphere}
in the current setting:
\be \label{eq-sigma-sigma0-S-z}
|| \sigma_r - \sigma_0 ||_{C^3 ({\mS^2}) } = O (r^{2} )
\ee
\be \label{eq-H-r-Ssphere-z}
H (r) = 2 - \frac{r^2}{3} \sum_{i=1}^3 \lambda_i x_i^2   +  O ( r^3 )
\ee
\be \label{eq-H-0-Ssphere-z}
H_0 (r) = 2  -  \frac{4r^2}{3} \sum_{i=1}^3 \lambda_i x_i^2  + O ( r^{ 3} )
\ee
and
\be \label{eq-BY-Ssphere-z}
\int_{\mS^2} ( H_0 (r)  - H (r) ) d v_{\sigma_r} =  \   \frac{\pi r^4}{15} \lf[ 2 \sum_{i=1}^3 \lambda_i^2   +  \Delta R (p) \ri] + O(r^5).
\ee

Given any constant $b$,   define
\be \label{eq-def-Hb}
H_{(b)}  (r) = H(r) + b r^4   \sum_{i=1}^3 \lambda_i^2   .
\ee
It follows from  \eqref{eq-Ric-lambda}, \eqref{eq-sigma-sigma0-S-z}, \eqref{eq-BY-Ssphere-z} and \eqref{eq-def-Hb} that
\be \label{eq-BY-Ssphere-z-b}
\int_{\mS^2} (H_0 (r)  - H_{(b)} (r)  ) d v_{\sigma_r} =  \   4 \pi r^4  \lf[ \frac{1}{30} -  \lf( b -  \frac{1}{60} \frac{ \Delta  R (p) }{| \Ric (p) |^2  }  \ri)  \ri] \sum_{i=1}^3 \lambda_i^2  + O(r^5).
\ee

We are in a position to state the main result in the remaining part of this paper ---
a  classification theorem on the positivity of $ F_{\sigma_r, H_{(b)} (r)} $, from which
Theorem \ref{thm-small-sphere-main}(b) will follow as a corollary.

\begin{thm} \label{thm-small-spheres-bz}
Under the above notations,
\begin{enumerate}
\item[(i)] if the constant
$$ \displaystyle  \lf(  b - \frac{1}{60} \frac{ \Delta R  (p) }{| \Ric (p) |^2} \ri) < \frac{1}{90},$$ then
there exist constants $ r_0 >0$  and $ C> 0 $ such that
for any $ 0 < r < r_0 $,
$$ F_{\sigma_r, H_{(b)}(r)} (\eta) \ge C r^4 \int_{\mS^2} (\Delta_r \eta)^2 \dvr, \ \ \forall \  \eta \in W^{2.2}(\mS^2) .$$
Here $\Delta_r$ denotes the Laplacian of the metric $\sigma_r$.

\item[(ii)]  if  the constant
$$ \displaystyle \lf(  b - \frac{1}{60} \frac{ \Delta R (p) }{| \Ric (p) |^2}  \ri) > \frac{1}{90} ,$$
 then there exists a constant $r_1 > 0 $ such that for any $ 0 < r < r_1$, there exists a function $\eta_r \in W^{2.2}(\mS^2)$ such that
$$ F_{\sigma_r, H_{(b)} (r)} (\eta_r ) < 0 . $$

\end{enumerate}
\end{thm}

\begin{proof}[Proof of Theorem \ref{thm-small-sphere-main}(b)]
Let $b = 0 $, the result follows from Theorem \ref{thm-small-spheres-bz} (i).
\end{proof}

The main ingredient in the proof of Theorem \ref{thm-small-spheres-bz}
is  the following result on $(\mS^2, \sigma_0)$.

\begin{prop}\label{prop-FQ-2}
Let $\sigma_0 $ be the standard metric on $\mS^2 = \{ | x | = 1 \} \subset  \R^3$.
Given  any three constants $\lambda_1, \lambda_2, \lambda_3$ satisfying
$$ \sum_{i=1}^3 \lambda_i = 0 \ {and} \   \sum_{i=1}^3 \lambda_i^2 > 0 , $$
define
$ \phi=\sum_{i=1}^3  \lambda_i x_i^2  . $
 Consider the functional
\bee
  \begin{split}
   G(\eta_1,\eta_2)=&4\pi   \lf(\frac1{30}- \bb\ri)\sumi\lambda_i^2 +\frac12  \int_{\mS^2} \eta_1^2 \phi^2 \dvo\\
&    -   2\int_{\mS^2}\phi \lf[  \frac{  (\Delta_0 \eta_1) ( \Delta_0\eta_2 )}{ 4} +\la\nabla_0\eta_1,\nabla_0\eta_2\ra\ri] \dvo\\
  &+\int_{\mS^2} \lf(\frac{\lf(\Delta_0\eta_2\ri)^2}2-|\nabla_0\eta_2|^2
  \ri)\dvo
  \end{split}
\eee
where  $ \bb$ is a constant, $\Delta_0$ and $ \nabla_0$ are the Laplacian and the gradient on $(\mS^2, \sigma_0)$,
and
  $\eta_1, \eta_2\in W^{2,2}(\mS^2)$ satisfy
\begin{itemize}
\item $ \eta_1= \sum_{i=1}^3 a_i x_i  $ for some  vector $a= (a_1, a_2, a_3) $ with $ |a| = 1$

\item  $\eta_2$ is $\sigma_0$-$L^2$ orthogonal to $\L (\sigma_0)$,
the space spanned by $\{ 1, x_1, x_2, x_3 \}$.
\end{itemize}
Then
\begin{enumerate}
\item[(i)] if $ \bb< \frac{1}{90}$,  $G(\eta_1, \eta_2) > 0 $ for any $ \eta_1$ and $\eta_2$.

\vh

\item[(ii)]  if $ \bb > \frac{1}{90} $,  given any $\eta_1 $, $\exists$  an $\eta_2 $ such that $ G(\eta_1, \eta_2) < 0 $.

\end{enumerate}

\end{prop}

\begin{proof}  Using the assumption $ \sumi \lambda_i = 0 $ and $ | a | = 1 $, it
is computed  in  Lemma \ref{lma-A-B} in the Appendix that
\bee
\begin{split}
 \int_{\mS^2} \eta_1^2 \phi^2 \dvo  = & \    \int_{\mS^2} \lf( \sumi a_i x_i \ri)^2 \lf( \sumi \lambda_i x_i^2 \ri)^2 \dvo \\
 & =  \ 16\pi\lf[\frac{2\sum_i a_i^2\lambda_i^2}{3\times 35}+\frac{ \sum_i \lambda_i^2}{2\times3\times 35}\ri] .
\end{split}
\eee
For simplicity, we write
\be \label{eq-expression-A}
 \displaystyle A = 16\pi\lf[\frac{2\sum_i a_i^2\lambda_i^2}{3\times 35}+\frac{ \sum_i \lambda_i^2}{2\times3\times 35}\ri]  .
 \ee

Next we define
\be
B = \int_{\mS^2}\phi \lf[  \frac{  (\Delta_0 \eta_1) ( \Delta_0\eta_2 )}{ 4} +\la\nabla_0\eta_1,\nabla_0\eta_2\ra\ri] \dvo
\ee
and claim
\bee
B = 10 \int_{\mS^2} \phi \eta_1 \eta_2 \dvo .
\eee
To see this, we note the following facts about $\phi$ and $\eta_1$:
\be
\Delta_0 \phi = - 6 \phi
\ee
where we used $\sumi \lambda_i = 0 $, and
\bee
\begin{split}
\la\nabla_0 \phi,\nabla_0  \eta_1\ra=&\la\ol\nabla\phi-\la\ol\nabla\phi, X \ra X ,\ol\nabla\eta_1 - \la\ol\nabla\eta_1, X \ra X \ra\\
=&\la\ol\nabla\phi, \ol\nabla\eta_1\ra-\la\ol\nabla\phi, X \ra\la\ol\nabla\eta_1, X \ra\\
=&2\sumi \lambda_i a_i x_i-2\phi\eta_1
\end{split}
\eee
where $\ol{\nabla} $ denotes the gradient on $\R^3$ and $ X = (x_1, x_2, x_3)$.
Now
\bee
\begin{split}
 B =  & \     \int_{\mS^2}\phi \lf[  \frac{  (\Delta_0 \eta_1) ( \Delta_0\eta_2 )}{ 4} +\la\nabla_0\eta_1,\nabla_0\eta_2\ra\ri] \dvo \\
=& \int_{\mathbb{S}^2}\lf[-\frac12\phi\eta_1\Delta_0 \eta_2+\la\nabla_0 \eta_1,
\nabla_0 (\phi\eta_2)\ra-\eta_2\la\nabla_0 \eta_1,
\nabla_0  \phi \ra \ri] dv_{\sigma_0} \\
=&\int_{\mathbb{S}^2}\lf[-\frac12\Delta_0 (\phi\eta_1)\eta_2-\phi\eta_2 \Delta_0 \eta_1
 -\eta_2\la\nabla_0 \eta_1,
\nabla_0  \phi \ra \ri] dv_{\sigma_0}\\
=&\int_{\mathbb{S}^2}\lf[4 \phi\eta_1 \eta_2+2\phi\eta_2 \eta_1
-2\eta_2 \la\nabla_0  \eta_1 ,  \nabla_0  \phi \ra \ri] dv_{\sigma_0}\\
=&10\int_{\mathbb{S}^2}  \phi\eta_1 \eta_2 dv_{\sigma_0} .
\end{split}
      \eee

To proceed, we let $\tau_2$ be the $L^2$ orthogonal projection of $\eta_2$ to the eigenspace of the second nonzero eigenvalue of $(\mS^2, \sigma_0)$
and let $\tau_3=\eta_2-\tau_2$. Then
\be \label{eq-eigen-est-C}
\begin{split}
& \ \int_{\mS^2}\lf(\frac{(\Delta_0\eta_2)^2}2-|\nabla_0\eta_2|^2
  \ri)\dvo\\
  =& \ \int_{\mS^2}\lf(\frac{(\Delta_0\tau_2)^2}2-|\nabla_0\tau_2|^2
  \ri)\dvo+\int_{\mS^2}\lf(\frac{(\Delta_0\tau_3)^2}2-|\nabla_0\tau_3|^2
  \ri)\dvo\\
  \ge& \ \frac13\int_{\mS^2}(\Delta_0\tau_2)^2\dvo+\frac{5}{12}\int_{\mS^2}
   (\Delta_0\tau_3)^2\dvo
\end{split}
\ee
where we have used  the assumption that $ \eta_2 $ is   $L^2$ orthogonal to $ \L (\sigma_0)$,
 Lemma  \ref{lma-Laplace-gradient}, and  the fact that the second and third nonzero eigenvalues of $(\mS^2,\sigma_0)$ are 6 and 12
respectively.

Note that  $\tau_2$ is the restriction to $\mS^2$ of a homogeneous polynomial  of degree two, hence $\phi \eta_1\tau_2$ is the restriction to $ \mS^2$ of a homogeneous polynomial  of degree five  which implies  $ \displaystyle \int_{\mS^2}\phi\eta_1\tau_2\dvo=0$.
Therefore
\be \label{eq-expression-B}
\begin{split}
B=10\int_{\mathbb{S}^2} \phi\eta_1\tau_3\dvo .
\end{split}
\ee
Now it follows from \eqref{eq-expression-A},  \eqref{eq-eigen-est-C} and \eqref{eq-expression-B} that
\be\label{eq-estofG}
\begin{split}
 G(\eta_1,\eta_2)\ge& \  \lf[4\pi \lf(\frac1{30}- \bb\ri)\sumi\lambda_i^2+\frac A2\ri]  - 2\times 10 \int_{\mathbb{S}^2} \phi\eta_1\tau_3\dvo \\
 & \ +\frac{5}{12}\int_{\mS^2} (\Delta_0\tau_3)^2 \dvo+
 \frac13\int_{\mS^2}(\Delta_0\tau_2)^2\dvo.
\end{split}\ee

Next, we make use of the fact  that  $\tau_3$ is $L^2$ orthogonal to $E_1$, the subspace spanned by $\{ x_1, x_2, x_3 \}$.  Therefore
$$
 \int_{\mathbb{S}^2} \phi\eta_1\tau_3\dvo=\int_{\mS^2}(\phi \eta_1-\xi)\tau_3 \dvo
$$
where $\xi$ is the $L^2$ orthogonal projection of $\phi \eta_1$ to $E_1$.  This implies
$$
|B| = 10 \lf| \int_{\mathbb{S}^2} \phi\eta_1\tau_3\dvo\ri| \le \frac56(A-\int_{\mS^2}\xi^2  dv_{\sigma_0})^\frac12\lf(\int_{\mS^2}(\Delta_0\tau_3 )^2 dv_{\sigma_0}\ri)^\frac12.
$$
To compute $\displaystyle \int_{\mS^2} \xi^2 \dvo $,
we have
\bee
\begin{split}
\int_{\mS^2} \phi \eta_1  x_1 dv_{\sigma_0}=& \int_{\mS^2}
(\lambda_1x_1^2+\lambda_2x_2^2+\lambda_3x_3^2) (a_1x_1^2+a_2x_1x_2+a_3x_1x_3) dv_{\sigma_0}\\
=&\int_{\mS^2}  a_1x_1^2
(\lambda_1x_1^2+\lambda_2x_2^2+\lambda_3x_3^2)dv_{\sigma_0}\\
=&4\pi a_1\lf(\frac{\lambda_1}5+\frac{\lambda_2}{15}+\frac{\lambda_3}{15}\ri)\\
=&\frac{8\pi a_1\lambda_1}{15}.
\end{split}
\eee
Similarly, for $i=2,3$,
$ \displaystyle
\int_{\mS^2}x_i\eta_1\phi dv_{\sigma_0}=\frac{8\pi b_i\lambda_i}{15} $.
Hence,
$
\xi= \frac25 \sum_{i=1}^3  a_i\lambda_i   x_i.
$
and
$$
\int_{\mS^2} ( \phi \eta_1 -  \xi )^2   dv_{\sigma_0}= A - \frac{16\pi}{75}\sum_ia_i^2\lambda_i^2.
$$
(In particular, this shows $ \phi_1 \eta_1 - \xi \neq 0 $ by \eqref{eq-expression-A}.)
Therefore,
\be \label{eq-refined-estofB}
|B|\le \frac{5}{6}\lf(A-\frac{16\pi}{75}\sum_{i=1}^3  a_i^2\lambda_i^2\ri)^\frac12\lf(\int_{\mS^2}(\Delta_0\tau_3 )^2 dv_{\sigma_0}\ri)^\frac12.
\ee
By  \eqref{eq-estofG} and \eqref{eq-refined-estofB}, we conclude that
\be \label{eq-estofG-1}
\begin{split}
 G(\eta_1,\eta_2) \ge &  \lf(\a  - 2\b t   + \gamma t^2 \ri) + \frac13\int_{\mS^2}(\Delta_0\tau_2)^2\dvo
  \end{split}
  \ee
where
$$
t= \lf(\int_{\mS^2}(\Delta_0\tau_3)^2\dvo\ri)^\frac12, \ \ \gamma=\frac5{12},
$$
$$
\a=4\pi \lf(\frac1{30}- \bb \ri)\sum_{i=1}^3   \lambda_i^2+\frac A2,
$$
$$
\b=\frac{5}{6}\lf(A-\frac{16\pi}{75}\sum_{i=1}^3  a_i^2\lambda_i^2\ri)^\frac12 .
$$
Direct calculation shows
 \be \label{eq-DELTA}
 \begin{split}
 \beta^2-\a\gamma  =  & \ \frac{25}{36}\lf(A-\frac{16\pi}{75}\sum_{i=1}^3  a_i^2\lambda_i^2\ri)-\frac5{12} \lf[ \lf(\frac{2\pi}{15} - 4 \pi \bb \ri)
 \sum_{i=1}^3 \lambda_i^2+\frac A2\ri]\\
 =&\frac{35}{72}A-\frac{4\pi}{27}\sum_{i=1}^3  a_i^2\lambda_i^2 -\frac{\pi}{18}  \sum_{i=1}^3 \lambda_i^2   +  \frac{5}{3}\pi \bb \sum_{i=1}^3 \lambda_i^2  \\
 =&\frac{35}{72}\cdot16\pi\lf[\frac{2\sumi a_i^2\lambda_i^2}{3\times 35}+\frac{ \sum_{i=1}^3  \lambda_i^2}{2\times3\times 35}\ri]-\frac{4\pi}{27}\sum_{i=1}^3   a_i^2\lambda_i^2-\frac{\pi}{18}\sumi \lambda_i^2   +  \frac{5}{3}\pi  \bb \sum_{i=1}^3   \lambda_i^2 \\
 =&\frac{4\pi}{27}\sumi a_i^2\lambda_i^2+\frac{\pi}{27}\sumi \lambda_i^2-\frac{4\pi}{27}\sumi a_i^2\lambda_i^2-\frac{\pi}{18}\sumi \lambda_i^2 +  \frac{5}{3}\pi \bb \sumi \lambda_i^2 \\
 =&- \lf( \frac{1}{54} -   \frac{5}{3} \bb \ri) \pi \sumi \lambda_i^2  .
 \end{split}
 \ee
 Therefore, if  $\displaystyle \bb < \frac{1}{90},$ by \eqref{eq-estofG-1} and \eqref{eq-DELTA} we have
\bee
\begin{split}
 G(\eta_1,\eta_2) \ge \frac{\alpha \gamma - \beta^2}{\gamma} =  4 \pi \lf( \frac{1}{90} -   \bb \ri) \sumi \lambda_i^2  > 0
\end{split}
\eee
which  proves (i).

To prove (ii), we  claim that the function $ \phi \eta_1 - \xi $ above is indeed an eigenfunction of the third nonzero eigenvalue $12$.
To verify this, we compute
\bee
\begin{split}
\Delta_0 (\phi \eta_1) = & \  (\Delta_0  \phi) \eta_1 + \phi (\Delta_0 \eta_1 ) + 2 \la \nabla_0 \phi, \nabla \eta_1 \ra \\
= & \  (-6) \phi \eta_1 + (-2) \phi \eta_1 + 2 \lf[ 2 \sumi a_i \lambda_i x_i - 2 \phi \eta_1 \ri] \\
= & \ (-12) \phi \eta_1 + 4 \sumi a_i \lambda_i x_i .
\end{split}
\eee
Therefore,
\bee
\begin{split}
\Delta_0 ( \phi \eta_1 - \xi )  = & \ (-12) \phi \eta_1 + 4 \sumi a_i \lambda_i x_i  + \frac{4}{5} \sumi   a_i \lambda_i x_i \\
= & \  (-12)  \lf( \phi \eta_1 -  \xi  \ri).
\end{split}
\eee

Now we fix an $a$ (hence $\eta_1$ is fixed), and let $ \eta_2 = k ( \phi \eta_1 - \xi ) $
where $ \xi $ is the defined above and $k$ is an arbitrary constant. Then
\bee \label{eq-estofG-2}
\begin{split}
 G(\eta_1,\eta_2)  = & \  \lf[4\pi \lf(\frac1{30}-b\ri)\sum_i\lambda_i^2+\frac A2\ri]  - 2 0 k   \int_{\mathbb{S}^2} ( \phi \eta_1 - \xi )^2 \dvo \\
 & \ +\frac{5}{12}  k^2 \int_{\mS^2} [\Delta_0 ( \phi \eta_1 - \xi )]^2 \dvo \\
 = & \ \alpha - 2 \beta t  + \gamma t^2 \\
 \end{split}
\eee
where $\alpha$, $\beta$ and $\gamma$ are defined as same as before and
$$ t =  12 \lf( A - \frac{16\pi}{75}\sum_ia_i^2\lambda_i^2 \ri)^\frac12 k . $$
Suppose $ \bb > \frac{1}{ 90} $,  it follows from \eqref{eq-DELTA}  that the above quadratic form of $t$ has
two distinctive roots. In particular, if  $k$ is chosen such that
$$ 12 \lf( A - \frac{16\pi}{75}\sum_ia_i^2\lambda_i^2 \ri)^\frac12 k = \frac{\beta}{\gamma} , $$
 then
\be  \label{eq-explicit-G}
G (\eta_1, k ( \phi \eta_1 - \xi) ) = \frac{\alpha \gamma - \beta^2}{\gamma} = 4 \pi \lf( \frac{1}{90} -   \bb \ri)  \sumi \lambda_i^2  < 0 .
\ee
This completes the proof.
\end{proof}

We are now ready to prove Theorem  \ref{thm-small-spheres-bz}. We first prove   part   (i):

\begin{proof}
Suppose (i) of Theorem \eqref{thm-small-spheres-bz} is not true, then there exist  two sequences of positive numbers $\{ r_k \}$, $\{ \e_k \}$ and
a sequence of functions $\{ \eta^{(k)} \} \subset W^{2,2} (\mS^2) $ such that
$$ r_k \rightarrow 0, \ \  \e_k \rightarrow 0 ,$$
\be\label{eq-thm-pf-2}
F_{\sigma_{r_k}, H_{(b)}(r_k)} (\eta^{(k)})< \e_kr_k^4\int_{\mS^2}(\Delta_{r_k}\eta^{(k)})^2dv_{\sigma_{r_k}}
\ee
and
\be \label{eq-average-etak}
\int_{\mS^2}\eta^{(k)}\dvk=0.
\ee
In the following, we  denote $\Delta_{r_k}$, $ \sigma_{r_k} $ by $\Delta_k$, $\sigma_k$ respectively.
We also let $\nabla_k $ denote the gradient on  $(\mS^2, \sigma_k)$.

On $\mS^2$, recall that $ \{ x_i \}$ are the restriction of the standard coordinate functions in $ \R^3$. Hence
$X_0 = (x_1, x_2, x_3) $ is an isometric embedding of $(\mS^2, \sigma_0)$.
By \eqref{eq-sigma-sigma0-S-z} and the result of Nirenberg (page 353 in \cite{Nirenberg}),
for each  large $k$, there exists an isometric embedding
$$
X_k  = (x_1^{(k)}, x_2^{(k)}, x_3^{(k)}): (\mS^2, \sigma_{k}) \longrightarrow \R^3
$$
satisfying
\be \label{eq-2nd-form-z}
|| x_i^{(k)} - x_i ||_{C^{2, \alpha} (\mS^2, \sigma_0) } = O (r_k^2) , \ \forall \ i = 1, 2, 3 .
\ee
Here, given an integer  $m$,  we use the notation
 $ O(r_k^m)$  to denote  some quantity $\psi$ satisfying   $ | \psi  | \le C r_k^m $ for a constant $C$ independent on $k $.
 Given such an $X_k$, we let  $\nu_0^{(k)} $ be the unit outward normal vector to $X_k (\mS^2)$ and
 $\Pi_0^{(k)} $ be the second fundamental form of $X_k (\mS^2)$ in $ \R^3$.
  It follows from
\eqref{eq-2nd-form-z} that
\be \label{eq-nu0-est}
|| \nu_0^{(k)} - X_0 ||_{C^{0,\alpha} (\mS^2, \sigma_0 )} = O (r_k^2)
\ee
and
\be \label{eq-Pi0-est}
|| \Pi_0^{(k)}  - \sigma_0 ||_{C^{0,\alpha} (\mS^2, \sigma_0)} = O (r_k^2) .
\ee

As before,  we let  $\L(\sigma_0) $ and $\L (\sigma_k)$ be  the subspaces of  $ W^{2,2}(\mS^2)$  which are spanned by
$\{ 1, x_1, x_2, x_3\}$ and  $\{1, x_1^{(k)}, x_2^{(k)}, x_3^{(k)} \}$ respectively.
For each $k$,  we decompose $\eta^{(k)}=\eta^{(k)}_1+\eta^{(k)}_2$, where
\bee
\eta^{(k)}_1 = a_0^{(k)} + \sumi a_i^{(k)}  x_i^{(k)} \in \mathcal{L}(\sigma_k)
\eee
and $\eta^{(k)}_2$ is $\sigma_k $-$L^2$ orthogonal to $\mathcal{L}(\sigma_k)$.
Let  $a^{(k)} = (a_1^{(k)}, a_2^{(k)}, a_3^{(k)} )$. Then
\be  \label{eq-F-etar}
\begin{split}
\ & \  F_{\sigma_k, H _{(b)} (r_k) } (\eta^{ (k) } ) \\
= & \ F_{\sigma_k, H_{(b)} (r_k) }  (\eta^{ (k) }_1) +
2 Q_{\sigma_k, H_{(b)} (r_k) }  (\eta^{ (k) }_1, \eta^{ (k) }_2 )
+  F_{\sigma_k, H_{(b)} (r_k) }  (\eta^{ (k) }_2 ) \\
\end{split}
\ee
where
\be \label{eq-Fetar1}
\begin{split}
F_{\sigma_k, H_{(b)} (r_k) }  (\eta^{ (k) }_1 )  = & \
 |a^{(k)} |^2 \int_{\mS^2} ( H_0 (r_k)    - H_{(b)} (r_k) ) \dvk  \\
 & \   + \int_{\mS^2}  \la a^{(k)} , \nu^{(k)}_0  \ra^2 \frac{ \lf( H_0 (r_k)  - H_{(b)} (r_k)  \ri)^2 }{H_{(b)}  (r_k)  }   \dvk ,
 \end{split}
\ee
\be \label{eq-Q-eta1-eta2}
\begin{split}
& \ Q_{\sigma_k, H_{(b)} (r_k) }  (\eta^{ (k) }_1, \eta^{ (k) }_2 )  \\
= & \ \int_{\mS^2}  \lf(  H_0 (r_k)  - H_{(b)} (r_k)   \ri)
\lf[  \frac{  ( \Delta_k \eta_1^{(k)}  ) ( \Delta_k  \eta_2^{(k)} )  }{   H_0 (r_k) H_{(b)} (r_k) }
+  \la \nabla_k  \eta_1^{(k)}  ,  \nabla_k  \eta_2^{(k)}  \ra_{\sigma_k}    \ri]  \dvk ,
\end{split}
\ee
and
\be \label{eq-F-etar2}
\begin{split}
& \  F_{\sigma_k, H_{(b)} (r_k) }  (\eta^{ (k) }_2 ) \\
=  & \  \int_{\mS^2} \lf[ \frac{ ( \Delta_k  \eta_2^{(k)}  )^2 }{ H_0 (r_k)  }  - \Pi_0^{(k)}  ( \nabla_k  \eta_2^{(k)}  ,\nabla_k \eta_2^{(k)} ) \ri] \dvk \\
& \ \  + \int_{\mS^2}    \lf( H_0 (r_k)  - H_{(b)} (r_k)  \ri)  \lf[   \frac{  ( \Delta_k  \eta_2^{(k)}  )^2  }{  H_0 (r_k) H_{(b)} (r_k)   }
+  |  \nabla_k  \eta_2^{(k)}  |^2_{\sigma_k}    \ri]  \dvk .
 \end{split}
\ee
By Lemma \ref{lem-supplement-1},   Lemma \ref{lma-Laplace-gradient},  \eqref{eq-sigma-sigma0-S-z} --
 \eqref{eq-H-0-Ssphere-z}  and \eqref{eq-def-Hb}, we have
 \be \label{eq-F-etar2-2}
 F_{\sigma_k, H_{(b)} (r_k) }  (\eta^{ (k) }_2 ) \ge   \lf[ \beta + O (r_k^2) \ri]   \int_{\mS^2} ( \Delta_k \eta^{(k)}_2 )^2 \dvk ,
 \ee
\be \label{eq-Q-eta1-eta2-2}
\begin{split}
& \  \lf| Q_{\sigma_k, H_{(b)} (r_k) }  (\eta^{ (k) }_1, \eta^{ (k) }_2 ) \ri|  \\
\le & \ C_1 r_{k}^2 \lf( \int_{\mS^2} ( \Delta_k \eta^{(k)}_1 )^2 \dvk \ri)^\frac12  \lf( \int_{\mS^2} ( \Delta_k \eta^{(k)}_2 )^2 \dvk \ri)^\frac12 .
\end{split}
\ee
  Here $C_1$ and $\beta  $ are some positive  constant  independent on $k$.

We normalize $\eta^{(k)}$ such that
\be\label{eq-thm-pf-3}
\int_{\mS^2}\lf[(\Delta_{k}\eta^{(k)}_1)^2+(\Delta_{k}\eta^{(k)}_2)^2\ri]\dvk=1.
\ee
Recall
\be \label{eq-formula-a-square}
\frac{( \Delta_k \eta_1^{(k)} )^2 }{H_0 (r_k)^2 } + | \nabla_k \eta_1^{(k)} |_{\sigma_k}^2 = | a^{(k)} |^2 .
\ee
Thus \eqref{eq-thm-pf-3} and \eqref{eq-formula-a-square},  together with Lemma \ref{lma-Laplace-gradient},
 \eqref{eq-sigma-sigma0-S-z} and \eqref{eq-H-0-Ssphere-z}, imply
that  there is a constant $C_2$ independent on $k$ such that
 \be \label{eq-upper-bd-a}
  | a^{(k)} | \le C_2 .
  \ee
It follows from  \eqref{eq-sigma-sigma0-S-z} - \eqref{eq-H-0-Ssphere-z},
 \eqref{eq-def-Hb}, \eqref{eq-BY-Ssphere-z-b},  \eqref{eq-2nd-form-z},  \eqref{eq-Fetar1} and  \eqref{eq-upper-bd-a}
that
\be \label{eq-F-etar-1-1}
| F_{\sigma_k, H_{(b)} (r_k) }  (\eta^{ (k) }_1 ) |  \le C_3    r_k^4
\ee
for some positive constants $C_3 $ independent on $k$.
By \eqref{eq-thm-pf-2}, \eqref{eq-F-etar}, \eqref{eq-Q-eta1-eta2-2}, \eqref{eq-thm-pf-3} and \eqref{eq-F-etar-1-1}, we then have
\bee
\lf[ \beta + O (r_k^2) \ri]   \int_{\mS^2} ( \Delta_k \eta^{(k)}_2 )^2 \dvk <  2 \e_kr_k^4
+  C_3     r_k^4  + 2 C_1  r_k^2
\eee
which  shows
\bee
\lim_{k \rightarrow \infty} \int_{\mS^2}(\Delta_k \eta_2^{(k)})^2dv_{\sigma_k} =  0 ,
\eee
and consequently
\bee
\lim_{k \rightarrow \infty} \int_{\mS^2}(\Delta_k \eta_1^{(k)})^2dv_{\sigma_k} =  1
\eee
by \eqref{eq-thm-pf-3}.
Therefore, for large $k$, by \eqref{eq-formula-a-square} we have
\be \label{eq-a-lower-bd}
| a^{(k)} | \ge  C_4
\ee
for some  positive constant $C_4$  independent on $k$.

Now  we renormalize $\eta^{(k)}$ such that $|a^{(k)}|=1$. By \eqref{eq-thm-pf-3} and \eqref{eq-a-lower-bd},
\be\label{eq-thm-pf-6}
\int_{\mS^2}\lf[(\Delta_{k}\eta^{(k)}_1)^2+(\Delta_{k}\eta^{(k)}_2)^2\ri]\dvk\le C_5
\ee
for some positive constant $C_5 $ independent of $k$.
Define
$$ \xi_k= \frac{\eta^{(k)}_2}{r_k^2} . $$
It follows from  \eqref{eq-sigma-sigma0-S-z} - \eqref{eq-H-0-Ssphere-z},  \eqref{eq-def-Hb},  \eqref{eq-BY-Ssphere-z-b}, \eqref{eq-thm-pf-2},
 \eqref{eq-2nd-form-z} - \eqref{eq-nu0-est} and  \eqref{eq-F-etar} - \eqref{eq-Fetar1} that
\be\label{eq-thm-pf-7}
\begin{split}
& \ \e_k\int_{\mS^2}(\Delta_{k}\eta^{(k)})^2dv_{\sigma_{k}}\\
\ge& \ 4\pi \lf[\lf(\frac1{30} - \bb\ri)\sum_i\lambda_i^2\ri]+\frac12 A_k  +O(r_k) \\
& \ + 2   r_k^{-2} Q_{\sigma_k, H_{(b)} (r_k) } (\eta_1^{(k)}, \xi_k ) + F_{\sigma_{k}, H_b(k)} (\xi_k) \\
\end{split}
\ee
where
\bee
\bb =  b - \frac{1}{60} \frac{ (\Delta_g R) (p) }{| \Ric (p) |^2}
\eee
and
\bee
\begin{split}
A_k   = & \    \int_{\mS^2} \lf( \sumi a_i^{(k)}  x_i \ri)^2 \lf( \sumi \lambda_i x_i^2 \ri)^2 \dvo \\
  =  & \ 16\pi\lf[\frac{2\sumi ( a_i^{(k)} )^2\lambda_i^2}{3\times 35}+\frac{ \sum_i \lambda_i^2}{2\times3\times 35}\ri]
\end{split}
\eee
(see the definition of $A$ in Proposition \ref{prop-FQ-2}).
Moreover, by \eqref{eq-F-etar2-2} and  \eqref{eq-Q-eta1-eta2-2}, we have
\be \label{eq-Q-eta1-eta2-2-2}
\begin{split}
& \  \lf|  r_k^{-2} Q_{\sigma_k, H_{(b)} (r_k) }  (\eta^{ (k) }_1, \xi_k ) \ri|  \\
\le & \ C_1\lf( \int_{\mS^2} ( \Delta_k \eta^{(k)}_1 )^2 \dvk \ri)^\frac12  \lf( \int_{\mS^2} ( \Delta_k \xi_k )^2 \dvk \ri)^\frac12
\end{split}
\ee
and
 \be \label{eq-F-etar2-2-2}
 F_{\sigma_k, H_{(b)} (r_k) }  ( \xi_k  ) \ge   \lf[ \beta + O (r_k^2) \ri]   \int_{\mS^2} ( \Delta_k \xi_k )^2 \dvk .
 \ee
 It follows from \eqref{eq-thm-pf-6} - \eqref{eq-F-etar2-2-2} that there exists a positive constant $ C_6 $ independent on $k$ such that
\be \label{eq-Lap-bound}
\int_{\mS^2}(\Delta_{k}\xi_k)^2dv_{\sigma_{k}}\le C_6 .
\ee

On the other hand, we still have  $ \xi_k  = r_k^{-2} \eta_2^{(k)} \in \L (\sigma_k )$. Hence,
\be \label{eq-xi-average}
\int_{\mS^2 } \xi_k \dvk = 0
\ee
and
\be \label{eq-xi-xi-average}
\int_{\mS^2} \xi_k x_i^{(k)} \dvk = 0 , \ \forall \ i = 1, 2, 3.
\ee
 By Lemma \ref{lma-Laplace-gradient}, \eqref{eq-sigma-sigma0-S-z}, \eqref{eq-Lap-bound}, \eqref{eq-xi-average}
 and the $L^2$-estimates, we know
\be \label{eq-W22-k-bound}
|| \xi_k ||_{W^{2,2} (\mS^2, \sigma_k)} \le  C_7
\ee
for some positive constant $C_7 $ independent on $k$.
This combined with \eqref{eq-sigma-sigma0-S-z} in turn shows
\be \label{eq-W22-0-bound}
|| \xi_k ||_{W^{2,2} (\mS^2, \sigma_0)} \le  C_8
\ee
for some positive constant $C_8 $ independent on $k$.
Therefore, there exists some $ \xi \in W^{2,2}(\mS^2)$ such that,
passing to a subsequence,
$\{ \xi_k \}$ converges  to $\xi $ weakly  in $W^{2,2}(\mS^2, \sigma_0)$
and strongly in $ W^{1,2} (\mS^2, \sigma_0) $.
Furthermore, it follows from   \eqref{eq-sigma-sigma0-S-z},  \eqref{eq-2nd-form-z}, \eqref{eq-xi-average} and \eqref{eq-xi-xi-average}
that  $ \xi \in \L (\sigma_0)$.

We will take limit in \eqref{eq-thm-pf-7}.  By  \eqref{eq-sigma-sigma0-S-z} - \eqref{eq-H-0-Ssphere-z},  \eqref{eq-def-Hb},
 \eqref{eq-2nd-form-z}, \eqref{eq-Pi0-est}, \eqref{eq-W22-k-bound} and \eqref{eq-W22-0-bound}, we have
\be \label{eq-F-etr-xi-k}
\begin{split}
&  \ F_{\sigma_{k}, H_{(b)}(r_k)} (\xi_k) \\
%=&\int_{\mS^2} \lf[ \frac{ ( \Delta_k \xi_k )^2 }{ H_{(b)}(r_k)  } + ( H_0(k) - H_{(b)}(r_k) ) | \nabla_k \xi_k  |^2_{\sigma_k}  - \Pi_0^{(k)} ( \nabla_k \xi_k ,\nabla_k \xi_k)\ri]\dvk\\
=&\int_{\mS^2} \lf[ \frac{ ( \Delta_k \xi_k )^2 }{ H_{(b)}(r_k)  } + ( H_0(k) - H_{(b)}(r_k) ) | \nabla_k \xi_k  |^2_{\sigma_k}  - \Pi^{(k)}_0 ( \nabla_k \xi_k ,\nabla_k \xi_k)\ri]\dvo+O(r_k^2)\\
=&\int_{\mS^2} \lf[ \frac{ ( \Delta_0\xi_k )^2 }{ 2 }     -   |\nabla_0 \xi_k|^2 \ri]\dvo+O(r_k^2) .
\end{split}
\ee
Since $ \{ \xi_k \} $ converges to $ \xi$ strongly in $W^{1,2}(\mS^2, \sigma_0)$ and
$ \{ \Delta_0 \xi_k \}$ converges to $ \Delta_0 \xi $  weakly in $L^2(\mS^2, \sigma_0)$,
\eqref{eq-F-etr-xi-k} implies
\be\label{eq-thm-pf-8}
\begin{split}
\liminf_{k\to\infty}F_{\sigma_{k}, H_{(b)}(k)} (\xi_k)\ge \int_{\mS^2} \lf[ \frac{ ( \Delta_0 \xi_0 )^2 }{ 2 }     -   |\nabla_0 \xi|^2 \ri]\dvo .
\end{split}
\ee

To take the limit of  $r_k^{-2} Q_{\sigma_k, H_{(b)} (r_k) } (\eta_1^{(k)}, \xi_k ) $, we can assume that
$ \{ a^{(k)} \} $ converges to some  $a  = (a_1, a_2, a_3) \in \mS^2 $ because $ | a^{(k)} | = 1 $.
By   \eqref{eq-H-r-Ssphere-z} \eqref{eq-H-0-Ssphere-z} and \eqref{eq-def-Hb}, we have
$
H_0 (r_k)-H_{(b)} (r_k)=-r_k^2\phi+O(r_k^3),
$
where
$ \displaystyle \phi=\sumi \lambda_ix_i^2 .$
 Similar to \eqref{eq-F-etr-xi-k},  we now have
\be\label{eq-thm-pf-9}
\begin{split}
 & \ r_k^{-2} Q_{\sigma_k, H_{(b)} (r_k) } (\eta_1^{(k)}, \xi_k )  \\
%= & \ \int_{\mS^2} r_k^{-2}( H_0(k)- H_{(b)}(k) )\lf[ \frac{ ( \Delta_k \eta^{(k)}_1 ) ( \Delta_k \xi_k )  }{ H_b(k)H_0(k)  }
%+   \la \nabla_k \eta^{(k)}_1 ,  \nabla_k \xi_k \ra_{\sigma_k}   \ri]\dvk\\
=& \ -\int_{\mS^2} \phi\lf[ \frac{ ( \Delta_k \eta^{(k)}_1 ) ( \Delta_k \xi_k )  }{ 4  }
 +   \la \nabla_k \eta^{(k)}_1 ,  \nabla_k \xi_k \ra_{\sigma_k}    \ri]\dvk+O(r_k)\\
=& \ -\int_{\mS^2} \phi\lf[ \frac{ ( \Delta_0 \eta^{(k)}_1 ) ( \Delta_0 \xi_k )  }{ 4  } +   \la \nabla_0 \eta^{(k)}_1 ,  \nabla_0 \xi_k \ra   \ri]\dvo+O(r_k)\\
=& \ -\int_{\mS^2} \phi\lf[ \frac{ \lf( \Delta_0 \lf(\sumi a_i^{(k)}x_i \ri)\ri) ( \Delta_0 \xi_k )  }{ 4  } +   \la \nabla_0 \lf(\sumi a_i^{(k)}x_i\ri) ,  \nabla_0 \xi_k \ra   \ri]\dvo+O(r_k)\\
=& \ -\int_{\mS^2} \lf[-\frac{ \la\nabla_0 \lf(\phi \Delta_0 \lf(\sumi a_i^{(k)}x_i \ri)\ri),\nabla_0 \xi_k \ra   }{ 4  } +   \phi\la \nabla_0 \lf(\sumi a_i^{(k)}x_i\ri) ,  \nabla_0 \xi_k \ra   \ri]\dvo+O(r_k)\\
\to & \  -\int_{\mS^2} \phi\lf[ \frac{ ( \Delta_0 \lf(\sumi a_ix_i\ri) ) ( \Delta_0 \xi )  }{ 4  } +   \la \nabla_0 \lf(\sumi a_ix_i\ri) ,  \nabla_0 \xi  \ra   \ri]\dvo,
\ \mathrm{as} \ k\to\infty
\end{split}
\ee
since $ \{ a^{(k) } \} $ converges to $a$ and $ \{ \xi_k \} $ converges to $ \xi$ strongly in $W^{1,2}(\mS^2, \sigma_0)$.

Combining \eqref{eq-thm-pf-6} - \eqref{eq-thm-pf-7} and  \eqref{eq-thm-pf-8} - \eqref{eq-thm-pf-9}, we conclude that
\be\label{eq-thm-pf-10}
\begin{split}
0 \ge & \ 4\pi \lf[\lf(\frac1{30} - \bb \ri)\sum_i\lambda_i^2\ri]+\frac 12 \int_{\mS^2} \lf( \sumi a_i  x_i \ri)^2 \phi^2 dv_{\sigma_0}  \\
&-2\int_{\mS^2} \phi\lf[ \frac{ ( \Delta_0 \lf(\sumi a_ix_i\ri) ) ( \Delta_0 \xi_k )  }{ 4  } +   \la \nabla_0 \lf(\sumi a_ix_i\ri) ,  \nabla_0 \xi  \ra   \ri]\dvo \\
& \  +\int_{\mS^2} \lf[ \frac{ ( \Delta_0 \xi_0 )^2 }{ 2 }     -   |\nabla_0 \xi|^2 \ri]\dvo .
\end{split}
\ee
Since
$$ | a | = 1, \ \  \xi \in \L(\sigma_0), \  \ \mathrm{and} \   \ \bb < \frac{1}{90}, $$
this  leads to a contradiction with
(i) of  Proposition \ref{prop-FQ-2}.
Therefore, (i) of Theorem \ref{thm-small-spheres-bz} is proved.
\end{proof}

Next, we prove (ii) of Theorem  \ref{thm-small-spheres-bz}.

\begin{proof}  Let
$ \displaystyle \bb =   b - \frac{1}{60} \frac{ (\Delta_g R) (p) }{| \Ric (p) |^2}   .$
Then
$ \bb > \frac{1}{90} $.  By (ii) of  Proposition \ref{prop-FQ-2}, given any
$ a = (a_1, a_2, a_3) $ with  $ | a | = 1 $,   there exists an $\eta_2 \in W^{2,2} (\mS^2)$
such that $ \eta_2 $ is $\sigma_0$-$L^2$ orthogonal to $ \L (\sigma_0) $
and
 \be \label{eq-G-eta1-eta2-T}
  \begin{split}
   G(\eta_1,\eta_2)= & \ 4\pi   \lf(\frac1{30}- \bb\ri)\sum_i\lambda_i^2 +\frac12  \int_{\mS^2} \eta_1^2 \phi^2 \dvo\\
&   \  -   2\int_{\mS^2}\phi \lf[ \frac{ (\Delta_0\eta_1) ( \Delta_0\eta_2 ) }4+\la\nabla_0\eta_1,\nabla_0\eta_2\ra\ri] \dvo\\
  & \ +\int_{\mS^2} \lf(\frac{\lf(\Delta_0\eta_2\ri)^2}2-|\nabla_0\eta_2|^2  \ri)\dvo  < 0 .
  \end{split}
\ee
Here  $ \eta_1 = \sumi a_i x_i  $ and $ \phi = \sumi \lambda_i x_i^2 . $

Let $ X_0 = (x_1, x_2, x_3) $.
For each small $r$,  let
$$
X_r  = (x_1^{(r)}, x_2^{(r)}, x_3^{(r)}): \mS^2 \longrightarrow \R^3
$$
be an isometric embedding of $(\mS^2, \sigma_r)$  satisfying
\be \label{eq-2nd-form-z-r}
|| x_i^{(r)} - x_i ||_{C^{2, \alpha} (\mS^2, \sigma_0) } = O (r^2) , \ \forall \ i = 1, 2, 3 .
\ee
 Let  $\nu_0^{(r)} $ be the unit outward normal vector to $X_r (\mS^2)$ and
 $\Pi_0^{(r)} $ be the second fundamental form of $X_r (\mS^2)$ in $ \R^3$.
  It follows from
\eqref{eq-2nd-form-z-r} that
\be \label{eq-nu0-est-r}
|| \nu_0^{(r)} - X_0 ||_{C^{0,\alpha} (\mS^2, \sigma_0)} = O (r^2)
\ee
and
\be \label{eq-Pi0-est-r}
|| \Pi_0^{(r)}  - \sigma_0 ||_{C^{0,\alpha} (\mS^2, \sigma_0)} = O (r^2) .
\ee

With $a $ and $\eta_2$ fixed,
we define
$$ \eta_1^{(r)} = \sumi a_i x_i^{(r)} \ \ \mathrm{and} \ \  \eta^{(r) } =  \eta_1^{(r)} + r^2 \eta_2 $$
for each small $r$. Then
\be  \label{eq-F-etar-r}
\begin{split}
\ & \  F_{\sigma_r, H _{(b)} (r) } (\eta^{ (r) } ) \\
= & \ F_{\sigma_r, H_{(b)} (r) }  (\eta_1^{(r)}) +
2 r^2 Q_{\sigma_r, H_{(b)} (r_r) }  (\eta_1^{(r)}, \eta_2 )
+  r^4 F_{\sigma_r, H_{(b)} (r) }  (\eta_2 ) .
\end{split}
\ee
We compare each term in  \eqref{eq-F-etar-r} with the corresponding term in \eqref{eq-G-eta1-eta2-T}.

First,
\be \label{eq-Fetar1-r}
\begin{split}
F_{\sigma_r, H_{(b)} (r) }  (\eta_1^{(r)} )  = & \
  \int_{\mS^2} ( H_0 (r)    - H_{(b)} (r) ) \dvr  \\
 & \   + \int_{\mS^2}  \la a, \nu^{(r)}_0  \ra^2 \frac{ \lf( H_0 (r)  - H_{(b)} (r)  \ri)^2 }{H_{(b)}  (r)  }   \dvr \\
 = & \ 4 \pi r^4  \lf( \frac{1}{30} - \bb \ri) \sum_{i=1}^3 \lambda_i^2  + O(r^5) \\
 & \  +  \frac12 r^4 \int_{\mS^2}  \eta_1^2  \phi^2   dv_{\sigma_0} + O (r^5)
\end{split}
 \ee
where we have used   \eqref{eq-sigma-sigma0-S-z} - \eqref{eq-H-0-Ssphere-z},
 \eqref{eq-def-Hb} - \eqref{eq-BY-Ssphere-z-b} and \eqref{eq-nu0-est-r}.

Next,  let $ \Delta_r $ and  $\nabla_r$ be the Laplacian and the gradient on $(\mS^2, \sigma_r)$ respectively.
Then
\be \label{eq-Q-eta1-eta2-r}
\begin{split}
& \ Q_{\sigma_r, H_{(b)} (r) }  (\eta_1^{(r)}, \eta_2 )  \\
= & \ \int_{\mS^2}  \lf(  H_0 (r)  - H_{(b)} (r)   \ri)
\lf[  \frac{  ( \Delta_r \eta_1^{(r)}  ) ( \Delta_r  \eta_2  )  }{   H_0 (r) H_{(b)} (r) }
+  \la \nabla_r  \eta_1^{(r)}  ,  \nabla_r  \eta_2 \ra_{\sigma_r}   \ri]  \dvr \\
= & \ \int_{\mS^2}  \lf( - \phi r^2 + O(r^3)  \ri)
\lf[  \frac{  ( \Delta_0 \eta_1  ) ( \Delta_0  \eta_2  )  + O(r^2)  }{  4 + O(r^2)  }
+  \la \nabla_0  \eta_1  ,  \nabla_0  \eta_2 \ra + O(r^2)      \ri]  ( 1 + O(r^2) \dvo \\
= & \ - r^2 \int_{\mS^2} \phi \lf[ \frac{\Delta_0 \eta_1 \Delta_0 \eta_2 }{4} + \la \nabla_0 \eta_1, \nabla_0 \eta_2 \ra \ri] \dvo + O(r^3)
\end{split}
\ee
where we have used  \eqref{eq-sigma-sigma0-S-z} - \eqref{eq-H-0-Ssphere-z},
 \eqref{eq-def-Hb} and \eqref{eq-2nd-form-z-r}.

Finally,
\be \label{eq-F-etar2-r}
\begin{split}
& \  F_{\sigma_r, H_{(b)} (r) }  (\eta_2 ) \\
=  & \  \int_{\mS^2} \lf[ \frac{ ( \Delta_r  \eta_2   )^2 }{ H_0 (r)  }  - \Pi_0^{(r)}  ( \nabla_r  \eta_2  ,\nabla_r  \eta_2 ) \ri] \dvr \\
& \ \  + \int_{\mS^2}    \lf( H_0 (r)  - H_{(b)} (r)  \ri)  \lf[   \frac{  ( \Delta_r  \eta_2  )^2  }{  H_0 (r) H_{(b)} (r)   }
+  |  \nabla_r  \eta_2  |^2_{\sigma_r}    \ri]  \dvr \\
= & \  \int_{\mS^2} \lf[ \frac{ ( \Delta_0  \eta_2)^2    + O(r^2)  }{ 2 + O(r^2)  }  - | \nabla_0 \eta_2 |^2 + O(r^2) \ri] \dvo  + O(r^2) \\
= & \  \int_{\mS^2} \lf[ \frac{ ( \Delta_0  \eta_2)^2   }{ 2  }  - | \nabla_0 \eta_2 |^2 \ri]  \dvo
 + O (r^2)
 \end{split}
\ee
where we have used  \eqref{eq-sigma-sigma0-S-z} - \eqref{eq-H-0-Ssphere-z},  \eqref{eq-def-Hb} and \eqref{eq-Pi0-est-r}.

It follows from \eqref{eq-F-etar-r} - \eqref{eq-F-etar2-r} that
\bee
F_{\sigma_r, H _{(b)} (r) } (\eta^{ (r) } )  = r^4 G(\eta_1, \eta_2) + O (r^5) .
\eee
Since $G(\eta_1, \eta_2) < 0 $, we conclude that there exists small $r_1 > 0$ such that
$ F_{\sigma_r, H _{(b)} (r) } (\eta^{ (r) } )  < 0 $ for any $ 0 < r < r_1$.
This completes the proof of  (ii) of Theorem \ref{thm-small-spheres-bz}.
\end{proof}

\section{Examples} \label{sect: no-fill-in}

We end this paper by  giving examples of positive functions $H$ on  $( \mS^2 , \sigma_0 )$
such that
\begin{itemize}
\item[(a)]  $  \int_{\mS^2} ( 2  - H) d v_{\sigma_0}  > 0 $

\vh

\item[(b)]
$ F_{\sigma_0, H} (\eta) < 0 $  {for  some} $\eta$.

\vh

\item[(c)] $||H-2||_{C^k(\mS^2, \sigma_0)} <  \e $ for any given $ \e > 0 $ and $ k \ge 2 $.

\end{itemize}
%Moreover, for any $\e>0$ and any $k\ge 2$, we can find such an $H$ so that $||H-2||_{C^k(\mS^2, \sigma_0)}\le \e.$
Such a function $H$ can be taken as one of  $H_{\bar{b}} (r) $ in the following.

\begin{thm}
Let $ \sigma_0 $ be the standard metric on $ \mS^2$.
Let $ \lambda_1, \lambda_2, \lambda_3$ be three constants satisfying
$ \sumi \lambda_i = 0 $ and $ \sumi \lambda_i^2 > 0$.
Define a $2$-parameter family of functions $\{ H_{\bar{b}} (r) \}$ on $ \mS^2$ by
\bee
\Hbbr = 2 +  r^2 \sumi \lambda_i x_i^2  -  \lf( \frac{1}{30} - \bar{b} \ri) r^4 \sumi \lambda_i^2
\eee
where
$ \bb \in \lf(  \frac{1}{90} , \frac{1}{30} \ri] $ {and}
$  r \in (0, \tilde{r}] .$
Here $\tilde{r} > 0 $ is any fixed constant such that $ \Hbbr $ is everywhere  positive.
(For instance,  $ \tilde{r}$  can be chosen such that
$
2  -   \tilde{r}^2 \sumi | \lambda_i |  -  \frac{1}{45} \tilde{r}^4 \sumi  \lambda_i^2  > 0 .
$)

Then
\begin{enumerate}
\item[(i)]  $  \int_{\mS^2} [ 2  - \Hbbr ] d v_{\sigma_0}  > 0 $, $ \forall  \ \bar{b} \in  \left(\frac{1}{90},  \frac{1}{30} \right)$
and $ \forall \ r \in (0, \tilde{r}]$.

\vh

\item[(ii)]  $\exists $  $ \tilde{C} >0 $ independent on $ \bar{b}$ such that
$ F_{\sigma_0, \Hbbr} (\eta) < 0 \ for \ some \ \eta $
 whenever
$  \bar{b} \in \left(\frac{1}{90},  \frac{1}{30} \right] \ and \
 0 < r < C  \lf( \frac{1}{90} -   \bb \ri)  \sumi \lambda_i^2  . $

 \end{enumerate}

\end{thm}

\begin{proof}
Since  $ \sumi \lambda_i = 0 $, $ \sumi \lambda_i^2 > 0$  and $ \bar{b} < \frac{1}{30} $, we have
\bee
\int_{\mS^2} [ 2  - \Hbbr ] \dvo =  4 \pi r^4 \lf( \frac{1}{30} - \bar{b} \ri)    \sumi \lambda_i^2  > 0
\eee
which proves (i).

Since $ \bar{b} > \frac{1}{90} $, by \eqref{eq-explicit-G}
we know    for any
$ a = (a_1, a_2, a_3) $ with  $ | a | = 1 $, there exists an $\eta_2 \in W^{2,2} (\mS^2)$
such that $ \eta_2 $ is $\sigma_0$-$L^2$ orthogonal to $ \L (\sigma_0) $
and
 \be \label{eq-G-eta1-eta2-T-end}
  \begin{split}
   G(\eta_1,\eta_2)= & \ 4\pi   \lf(\frac1{30}- \bb\ri)\sum_i\lambda_i^2 +\frac12  \int_{\mS^2} \eta_1^2 \phi^2 \dvo\\
&   \  -   2\int_{\mS^2}\phi \lf[ \frac{ (\Delta_0\eta_1) ( \Delta_0\eta_2 ) }4+\la\nabla_0\eta_1,\nabla_0\eta_2\ra\ri] \dvo\\
  & \ +\int_{\mS^2} \lf(\frac{\lf(\Delta_0\eta_2\ri)^2}2-|\nabla_0\eta_2|^2  \ri)\dvo  \\
  = & \  4 \pi \lf( \frac{1}{90} -   \bb \ri)  \sumi \lambda_i^2  < 0
  \end{split}
  \ee
where  $ \eta_1 = \sumi a_i x_i  $ and $ \phi = \sumi \lambda_i x_i^2 . $

With such $ \eta_1 $ and $\eta_2$ fixed, for each small $r$,
define
$$ \eta^{(r) } =  \eta_1 + r^2 \eta_2 .$$
Similar to the proof of (ii) of Theorem \ref{thm-small-spheres-bz}, we have
\be  \label{eq-F-etar-r-S5}
\begin{split}
\ & \  F_{\sigma_0, \Hbbr } (\eta^{ (r )}) = F_{\sigma_0, \Hbbr }  (\eta_1) +
2 r^2 Q_{\sigma_0, \Hbbr }  (\eta_1, \eta_2 )
+  r^4 F_{\sigma_0, \Hbbr }  (\eta_2 )
\end{split}
\ee
where
\be \label{eq-Fetar1-r-S5}
\begin{split}
& \ F_{\sigma_0, \Hbbr }  (\eta_1  )  \\
= & \   \int_{\mS^2} [ 2  - \Hbbr ] \dvo
   + \int_{\mS^2}  \eta_1^2 \frac{ \lf[ 2  - \Hbbr  \ri]^2 }{ \Hbbr }   \dvo \\
 = & \ 4 \pi r^4  \lf( \frac{1}{30} - \bb \ri) \sum_{i=1}^3 \lambda_i^2
   +  \frac12 r^4 \int_{\mS^2}  \eta_1^2  \phi^2   dv_{\sigma_0} + O (r^5) ,
\end{split}
 \ee
\be \label{eq-Q-eta1-eta2-r-S5}
\begin{split}
& \ Q_{\sigma_0, \Hbbr }  (\eta_1 , \eta_2 )  \\
= & \ \int_{\mS^2}  \lf[  2   - \Hbbr   \ri]
\lf[  \frac{  ( \Delta_0 \eta_1   ) ( \Delta_0  \eta_2  )  }{   2  \Hbbr }
+  \la \nabla_0  \eta_1  ,  \nabla_0  \eta_2 \ra_{\sigma_0}   \ri]  \dvo \\
= & \ - r^2 \int_{\mS^2} \phi \lf[ \frac{\Delta_0 \eta_1 \Delta_0 \eta_2 }{4} + \la \nabla_0 \eta_1, \nabla_0 \eta_2 \ra \ri] \dvo + O(r^3)
\end{split}
\ee
and
\be \label{eq-F-etar2-r-S5}
\begin{split}
& \  F_{\sigma_0, \Hbbr ) }  (\eta_2 ) \\
=  & \  \int_{\mS^2} \lf[ \frac{ ( \Delta_0  \eta_2   )^2 }{ 2  } - | \nabla_0 \eta_2 |^2 \ri] \dvo \\
& \ \  + \int_{\mS^2}    \lf[  2  - \Hbbr  \ri]  \lf[   \frac{  ( \Delta_0 \eta_2  )^2  }{  2 \Hbbr   }
+  |  \nabla_0  \eta_2  |^2_{\sigma_0}    \ri]  \dvo \\
= & \  \int_{\mS^2} \lf[ \frac{ ( \Delta_0  \eta_2)^2   }{ 2  }  - | \nabla_0 \eta_2 |^2 \ri]  \dvo
 + O (r^2).
 \end{split}
\ee
Here it is important to note that $ O(r^k)$  denotes a quantity $f$ that satisfies
 $ | f | \le C r^k $ for some constant $ C $ independent on $ \bar{b} \in \lf( \frac{1}{90}, \frac{1}{30} \ri)$.

It follows from \eqref{eq-G-eta1-eta2-T-end} - \eqref{eq-F-etar2-r-S5} that
\be
\begin{split}
F_{\sigma_r, \Hbbr } (\eta^{ (r )}) = & \ r^4 G(\eta_1, \eta_2) + O (r^5) \\
\le & \ r^4 \lf[  4 \pi \lf( \frac{1}{90} -   \bb \ri)  \sumi \lambda_i^2  + C r  \ri]  \\
\end{split}
\ee
for some constant $ {C} > 0 $ independent on $ \bar{b}$.
Therefore,
$$  F_{\sigma_r, \Hbbr } (\eta^{ (r )})   < 0$$
whenever $ 0 < r < C  4 \pi \lf( \bb -  \frac{1}{90}\ri)  \sumi \lambda_i^2 $.
This proves (ii).
\end{proof}

\section{Appendix}

\begin{lemma}  \label{lem-xixjxk}
Let $\sigma_0 $ be the standard metric on $\mS^2 = \{ | x | = 1 \} $ in  $ \R^3$. Then
\be
\begin{split}
 & \ \int_{\mS^2}x_1^{2k} \dvo =\frac{4\pi}{2k+1}; \
\int_{\mS^2}x_1^2x_2^2 \dvo = \frac{4\pi}{15}; \\
& \ \int_{\mS^2}x_1^{4}x_2^2 \dvo = \frac{4\pi}{35}; \
\int_{\mS^2}x_1^{2}x_2^2x_3^2 \dvo = \frac{4\pi}{3\times35}. \\
\end{split}
\ee
\end{lemma}
\begin{proof} The first integral follows directly from integration using polar coordinates.
To verify the second and the third integral, one has
\bee
\begin{split}
\int_{\mS^2}x_1^2x_2^2 \dvo  =  & \ \frac12 \int_{\mS^2}x_1^2 ( x_2^2 + x_3^2)  \dvo  \\
= & \ \frac12 \int_{\mS^2}x_1^2 ( 1 - x_1^2 )  \dvo \\
= & \ \frac{4\pi}{15}
\end{split}
\eee
and
\bee
\begin{split}
\int_{\mS^2}x_1^{4}x_2^2 \dvo =&\frac12\int_{\mS^2}x_1^4(x_2^2+x_3^2) \dvo \\
 =&\frac12\int_{\mS^2}x_1^4(1-x_1^2) \dvo \\
 =&\frac{4\pi}{35}
\end{split}
\eee
using  the first integral. To check the fourth integral, one notes that
\bee
\begin{split}
\frac{4\pi}3=&\int_{\mS^2}x_1^2 \dvo \\
=& \int_{\mS^2}x_1^2(x_1^2+x_2^2+x_3^2)^2  \dvo \\
 =&\int_{\mS^2}x_1^2(x_1^4+x_2^4+x_3^4+2x_1^2x_2^2+2x_1^2x_3^2+2x_2^2x_3^2)  \dvo  \\
 =&4\pi\lf( \frac17+\frac1{35}+\frac1{35}+\frac2{35}+\frac2{35}\ri)+2\int_{\mS^2}x_1^2x_2^2x_3^2 \dvo
\end{split}
\eee
So
\bee
2\int_{\mS^2}x_1^2x_2^2x_3^2 \dvo =4\pi\lf(\frac13-\frac{11}{35}\ri)=\frac{8 \pi} {3\times 35}.
\eee
\end{proof}

\begin{lemma} \label{lma-A-B}
Let $\sigma_0 $ be the standard metric on $\mS^2 = \{ | x | = 1 \} $ in  $ \R^3$.
Let $a_1, a_2, a_2 $ be three constants satisfying $\sumi a_i^2  = 1 $
and
 $\lambda_1, \lambda_2, \lambda_3$ be three constants satisfying
$ \sum_{i=1}^3 \lambda_i = 0 $.
Then
 $$
  \int_{\mathbb{S}^2} \lf( \sumi a_i x_i \ri)^2 \lf( \sumi \lambda_i x_i^2 \ri)^2 dv_{\sigma_0}=16\pi\lf[\frac{2\sum_i a_i^2\lambda_i^2}{3\times 35}+\frac{ \sum_i \lambda_i^2}{2\times3\times 35}\ri] .$$
\end{lemma}

\begin{proof}
Let $ \displaystyle A =  \int_{\mathbb{S}^2} \lf( \sumi a_i x_i \ri)^2 \lf( \sumi \lambda_i x_i^2 \ri)^2 dv_{\sigma_0}$.
By Lemma \ref{lem-xixjxk},
 \bee
 \begin{split}
            & \    \int_{\mathbb{S}^2}x_1^2  \lf(\sum_i\lambda_ix_i^2 \ri)^2   dv_{\sigma_0} \\
               = \ & \int_{\mathbb{S}^2}  x_1^2\lf(\sum_{i}\lambda_i^2x_i^4+2\lambda_1\lambda_2
               x_1^2x_2^2+2\lambda_1\lambda_3x_1^2x_3^2 +2\lambda_2\lambda_3x_2^2x_3^2\ri)  dv_{\sigma_0}\\
               = \ & {4\pi} \lf(\frac{\lambda_1^2}7+\frac{\lambda_2^2}{35}+
               \frac{\lambda_3^2}{35}+\frac{2\lambda_1\lambda_2}{35}+\frac{2\lambda_1\lambda_3}{35}
               +\frac{2\lambda_2\lambda_3}{3\times35}\ri)\\
               = \ & 4\pi \lf(\frac{\lambda_1^2}7-\frac{2\lambda_1^2}{35}+
               \frac1{35}(\lambda_2+\lambda_3)^2-\frac{4\lambda_2\lambda_3}{3\times35}\ri)\\
               = \ & 16\pi\lf(\frac{ \lambda_1^2}{35}-\frac{ \lambda_2\lambda_3}{3\times35}\ri)\\
               = \ &16\pi\lf(\frac{ \lambda_1^2}{35}-\frac{ \lambda_1^2- \lambda_2^2-\lambda_3^2}{2\times3\times35}\ri)
               \end{split}
               \eee
where one   uses  the fact
$2\lambda_2\lambda_3=(\lambda_2+\lambda_3)^2- \lambda_2^2-\lambda_3^2=\lambda_1^2- \lambda_2^2-\lambda_3^2 .$
               Similarly,
               \bee
               \int_{\mathbb{S}^2}x_2^2  \lf(\sum_i\lambda_ix_i^2\ri)^2   dv_{\sigma_0}=16\pi\lf(\frac{ \lambda_2^2}{35}-\frac{ \lambda_2^2- \lambda_3^2-\lambda_1^2}{2\times3\times35}\ri),
               \eee

           \bee
               \int_{\mathbb{S}^2}x_3^2  \lf(\sum_i\lambda_ix_i^2 \ri)^2   dv_{\sigma_0}=16\pi\lf(\frac{\lambda_3^2}{35}-\frac{ \lambda_3^2- \lambda_1^2-\lambda_2^2}{2\times3\times35}\ri).
               \eee
    On the other hand,
    $ \displaystyle
    \int_{\mathbb{S}^2}x_i x_j    \lf(\sum_i\lambda_ix_i^2 \ri)^2  dv_{\sigma_0}=0
    $ for $ \forall \ i \neq j $.
Hence, using the fact $ \sumi a_i^2 =1$, one concludes
    \bee
    \begin{split}
    A=& \ \frac{16\pi}{35}\lf[\sumi  a_i^2 \lambda_i^2+\frac1{2\times 3}\lf(-\sumi a_i^2\lambda_i^2+a_1^2(\lambda_2^2+\lambda_3^3)
    +a_2^2(\lambda_3^2+\lambda_1^3)+a_3^2(\lambda_1^2+\lambda_2^3)\ri)\ri]\\
    =& \ \frac{16\pi}{35}\lf[\sumi a_i^2\lambda_i^2+\frac1{2\times 3}\lf(-2\sumi a_i^2\lambda_i^2+\sumi \lambda_i^2\ri)\ri]\\
    =& \ 16\pi\lf[\frac{2\sumi a_i^2\lambda_i^2}{3\times 35}+\frac{ \sumi \lambda_i^2}{2\times3\times 35}\ri] .
    \end{split}
    \eee
\end{proof}

\end{document}